\documentclass[12pt,a4paper,reqno]{amsart}
\usepackage[dvipsnames]{xcolor}
\usepackage{amssymb}
\usepackage{amscd}
\usepackage{hyperref,url}
\usepackage{enumerate}
\usepackage{algorithm}
\usepackage{algpseudocode}
\usepackage{graphicx}
\usepackage{mathtools}
\usepackage{siunitx}
\usepackage{tikz-cd}
\usepackage{bm}
\usepackage{subfigure}
\numberwithin{equation}{section}
\usepackage{booktabs}
\usepackage{setspace}
\usepackage[charsperline=100]{jlcode}

\usepackage[obeyFinal,textsize=footnotesize]{todonotes}

\usepackage{mathtools}
\usepackage[tableposition=top]{caption}
\usepackage{booktabs,dcolumn}

\DeclareFontFamily{OT1}{rsfs}{}
\DeclareFontShape{OT1}{rsfs}{n}{it}{<-> rsfs10}{}
\DeclareMathAlphabet{\mathscr}{OT1}{rsfs}{n}{it}

\theoremstyle{plain}

\newtheorem{theorem}{Theorem}[section]
\newtheorem{proposition}[theorem]{Proposition}

\newtheorem{lemma}[theorem]{Lemma}
\newtheorem{corollary}[theorem]{Corollary}
\newtheorem{conjecture}[theorem]{Conjecture}

\theoremstyle{definition}

\renewcommand\H{\mathbb{H}}

\renewcommand\r{\mathbf{r}}

\begin{document}

\title[Maximum Growth Factor in Gaussian Elimination]{Some New Results on the Maximum Growth Factor in Gaussian Elimination}
\author{Alan Edelman}
\address{Department of Mathematics, Massachusetts Institute of Technology, Cambridge, MA, 02139 USA.
}
\email{edelman@mit.edu}
\author{John Urschel}
\address{Department of Mathematics, Massachusetts Institute of Technology, Cambridge, MA, 02139 USA.
}
\email{urschel@mit.edu}
\subjclass[2020]{Primary 65F05, 15A23.}
\thanks{ {\rm Accompanying software and data may be found in the \href{https://github.com/alanedelman/CompletePivotingGrowth}{online repository.}}}

\begin{abstract}
This paper combines modern numerical computation with theoretical results to improve our understanding of the
growth factor problem for Gaussian elimination.  On the computational side we obtain lower bounds for the maximum
growth for complete pivoting for $n=1:75$  and $n=100$ using the Julia JuMP optimization package.  At $n=100$
we obtain a growth factor bigger than $3n$. 
The numerical evidence suggests that the maximum growth factor is bigger than $n$ if and only if $n \ge 11$.

We also present a number of theoretical results. We show that the maximum growth factor over matrices with entries restricted to a subset of the reals is nearly equal to the maximum growth factor over all real matrices.
We also show that the growth factors under floating point arithmetic and exact arithmetic
are nearly identical. Finally, through numerical search, and stability and extrapolation results, we provide
improved lower bounds for the maximum growth factor. Specifically, we find that the largest growth factor is bigger than $1.0045n$ for $n>10$,
and the lim sup of the ratio with $n$ is greater than or equal to $3.317$.  
In contrast to the old conjecture that growth might never  be bigger than $n$, it seems likely that the maximum
growth divided by $n$ goes to infinity as $n \rightarrow \infty$.

\end{abstract}

\maketitle

\section{Introduction}\label{sec:intro}

We begin with a sketch of the history of the subject. For an introduction to the technical background of the field, and a more technical discussion of related work, see Subsections \ref{sub:technical} and \ref{subs:relatedtech}, respectively.

\subsection{History of Complete Pivoting (Overview) }
Understanding  and bounding the growth factor for Gaussian elimination has 
intrigued mathematical numerical linear algebraists for many decades.  It is one of those
beautiful problems that is quite easy to state, and yet we still know so very little,
especially in the context of complete pivoting.  This may seem even more surprising
as the problem has been around for more than six decades.

\noindent {\bf The 1960s-1970s:}
In 1961, Wilkinson \cite[Equation 4.15]{wilkinson1961error} published a famous
bound  for the growth factor for complete pivoting which was always considered pessimistic  (see Eq.\ (\ref{eqn:wilk}) of this paper).  In that same paper he writes 
``no matrix has been encountered for which [the growth factor] was as large as 8."

In 1964, Leonard Tornheim wrote 
in a technical paper  \cite{tornheim1964maximum}
``there is a conjecture, attributed
to J.H. Wilkinson, that" the growth factor for complete pivoting would
be at most $n$.    This is the first written record of this
conjecture (known to be false since 1991) as far as we are aware.
Leonard Tornheim's work is notable in that it was not published in any
journals, but rather as technical papers and Chevron research reports during a time
when Chevron had a major presence in Richmond, CA, some 20 miles from UC Berkeley.\footnote{Tornheim was a rather active mathematician. In a March 2023 phone call, the first author contacted
Tornheim's son who described how Tornheim wrote a computer program
that made Chevron ten million dollars, thus establishing in his words
``what a mathematician could be good for" (in private industry).
He also mentioned that his father had moved to the east coast and as 
coincidence would have it, he 
lived directly across the street
in Brookline, MA but sadly, the first author did not
know this and  recognizes what a lost opportunity this might have been.
}
One year later,
in 1965, Wilkinson 
published the famous book {\it The Algebraic Eigenvalue Problem}
\cite[p.213]{Wilkinson1965AEP} in which he wrote that no matrix had yet been discovered with growth greater than $n$
for an $n \times n$ matrix.
In that same year a cover letter written by Eugene B. Reid\footnote{In March 2023, the first author telephoned the son of Eugene B. Reid who described his father
as having bought the first commercial computer on the West Coast of the US,
and that his father was not a mathematician or a statistician, but 
he was responsible for the computer, mathematical, and statistical activities
at Chevron.  He also bought Chevron's very first computer which his son
claims might have been the first commercial computer on the West Coast
of the US.
}
for a report by Tornheim \cite{tornheim1965maximum}
writes (without any reference) that it was a  {\em widely known conjecture}
that growth would always be less than or equal to $n$.
In 1968, Cryer  \cite{cryer1968pivot} specifically references Wilkinson's 1965 words
as a conjecture,
though  Wilkinson 
never explicitly put to print a formal conjecture. 
Nonetheless Cryer specifically wrote in his abstract ``It has been conjectured by Wilkinson\ldots".   A possible best guess is that through the rumor grape vine often known as the children's game of ``broken telephone," the observation that none
was ever seen morphed into a folk conjecture
and thus we will attribute this conjecture not to Wilkinson (who did not write it down) nor Cryer (who wrote it later), but to folklore.\footnote{
In another March 2023 phone call, the first author spoke to William Kahan who
seemed rather certain that Wilkinson had never stated the conjecture
in so many words.  Kahan described how computers were slow enough
at the time that Wilkinson would watch the bits go by and look
for large growth, and growth $>n$ was never observed.
Nonetheless, we can not completely rule out that Wilkinson may have
verbally stated or at least hinted at the conjecture.
In fact, Cleve Moler stated publicly on
August 16, 2023 in Oxford, UK that Wilkinson
was not inclined to use terms such as
``theorem," ``proof,"
and ``conjecture,"
but Moler felt
that Wilkinson had believed the conjecture.
}

The 1960s saw the maximum growth computed for $n=1,2,3,4$  and bounded for $n=5$
by the growth chasers Tornheim  \cite{tornheim1964maximum,tornheim1965maximum,
tornheim1969maximum,tornheim1970maximum}, Cryer
\cite{cryer1968pivot}, and Cohen \cite{cohen1974note}. Hadamard matrices
(matrices with all entries $\pm1$ and orthogonal columns)
were shown
by Tornheim \cite{tornheim1964maximum}  and Cryer \cite{cryer1968pivot}   to have last pivot $n$.  
It became natural to wonder whether Hadamard matrices could be a counterexample
to the Wilkinson conjecture.

\noindent {\bf The 1980s:} After a bit of a lull, in 1988 Day and Peterson wrote a lovely article in
 the American Math Monthly \cite{day1988growth}  which revisits $n=3$, studies Hadamard matrices,
 and notably is the first to explore the growth problem for complete pivoting with numerical optimization software,
 specifically the NPSOL Library out of Stanford (Nonlinear Programming, Stanford Optimization Laboratory).
In particular they were the first to observe the number  $4.1325$ when $n=5$ as an output of the optimization software.
In 1989,  Higham and Higham \cite{higham1989large} pointed out that many common matrices can have 
growth factors of order $n$ (for any pivoting strategy).

 \noindent {\bf The early 1990s:} Interest in the growth factor
was substantially rejuvenated when
Trefethen and Schreiber  \cite{trefethen1990average} performed average case analyses
of the growth factor in 1990.  One year later, Nick Gould  \cite{gould1991growth} surprised everyone by  finding a 13x13 matrix with growth
bigger than 13 in finite precision using his LANCELOT software.   The solution was confirmed to be near a true example in 
exact arithmetic in 1992 \cite{edelman1992complete}.

\noindent{\bf 1993-Present:} In the over 30 years since, there was no progress whatsoever in improving Gould's numbers
for complete pivoting
through computation  (which would raise a lower bound) or lowering any mathematical upper bounds.
This is a testament to the difficulty of the problem.

\subsection{Other pivoting analyses}

\hfill\\

\noindent{\bf No Pivoting:}
In 2006,
the celebrated smoothed analysis\footnote{ Incidentally  ``smoothed analysis" was named by the first author in his car while driving Dan Spielman and Shanghua Teng in Cambridge, MA.} of Sankar, Spielman, and Teng \cite{sankar2006smoothed} showed that
large growth is unlikely from a probabilistic perturbative viewpoint with no pivoting, and pointed out that such an analysis could be possible for partial and complete
pivoting.

\noindent{\bf Partial Pivoting:} 
In 1994, Foster \cite{Foster:1994:GEP} pointed out that practical problems can bump into the unacceptable $2^{n-1}$ bound
for partial pivoting.  The first author remarked in \cite{edelman1995complete} that numerical experiments suggested in contrast
to  \cite{trefethen1990average} that the growth might be more like $O(n^{1/2})$ than $O(n^{2/3})$ on average. In addition to the smoothed analysis for no pivoting, Sankar \cite{sankar2004smoothed} also performed a smoothed analysis of partial pivoting with sub-exponential bounds.
Very recently  Huang and Tikhomirov \cite{huang2022average}  obtained new results exploring the average case analysis for partial pivoting.

\noindent{\bf Complete Pivoting for Hadamard Matrices:}  
It remains unknown, though perhaps it seems unlikely,  that a Hadamard matrix could have an earlier pivot bigger than $n$, given that the last
three pivots can only  be $n/2,n/2$ and $n$, and the fourth from the end is at most $n/2$.
Nonetheless, complete pivot patterns for Hadamard matrices remain a fascinating topic of research.
A comprehensive review of the topic including new progress written in 2013 by Kravvaritis may be found in \cite{kravvaritis2013complete}.
Of note are the investigations by Seberry \cite{seberry} and also \cite{edelman1998counterexample, edelman1995complete}. We note that the growth factor for a Hadamard matrix of dimension $n \le 16$ is known to be $n$ under complete pivoting.

\subsection{Technical Background}\label{sub:technical}

The solution of a linear system, i.e., given a matrix $A$ and vector $b$, finding a vector $x$ satisfying $Ax = b$, is one of the oldest problems in mathematics. Gaussian elimination, a technique in which a matrix is factored into the product of a lower and upper triangular matrix, is one of the most fundamental and important techniques for solving linear systems. The algorithm proceeds by converting $A$ into upper triangular form through row operations. In particular, given an $n \times n$ matrix $A = (a_{i,j})$, Gaussian elimination performs the iteration
\begin{align*}
    a_{i,j}^{(1)}&:= a_{i,j} \qquad \qquad \qquad \quad \text{for} \quad i,j = 1,...,n, \\
    a^{(k+1)}_{i,j} &:= a^{(k)}_{i,j} - \frac{a^{(k)}_{i,k} a^{(k)}_{k,j}}{a^{(k)}_{k,k}} \qquad \text{for} \quad   i,j = k,...,n, \; k = 1,...,n-1.
\end{align*}
This can be equivalently written as successive rank one updates of sub-matrices of $A$, i.e.,
$$ A^{(k+1)} := A^{(k)}_{k+1:n,k+1:n} - \frac{1}{a_{k,k}^{(k)}}\, A^{(k)}_{k+1:n,k} \, A^{(k)}_{k,k+1:n} \qquad \text{for} \quad k = 1,...,n-1,$$
where $A^{(k)} = (a_{i,j}^{(k)})_{i,j\ge k}$ and $A_{i_1:i_2,j_1:j_2}$ is defined as the sub-matrix of $A$ containing only rows $\{i_1,...,i_2\}$ and columns $\{j_1,...,j_2\}$. The resulting LU factorization of $A$ is given by
$$L(i,j) = \frac{a_{i,j}^{(j)}}{a_{j,j}^{(j)}} \quad \text{for} \quad i \ge j, \quad \text{and} \quad U(i,j) = a_{i,j}^{(i)} \quad \text{for} \quad j \ge i,$$
and this factorization is unique (up to scaling, i.e., $A =(LD)(D^{-1}U)$ for any invertible diagonal matrix $D$). Not all matrices have an LU factorization (issues arise if $a^{(k)}_{k,k} = 0$ for some $k<n$), and may require a permutation of the rows (or, equivalently, columns) of the matrix in order for such a factorization to exist. In addition, when computations are performed in finite precision, issues due to round-off error can occur. The backward error due to rounding in Gaussian elimination can be estimated by the number of bits of precision, the condition number of the matrix $A$, and the growth factor of the Gaussian elimination algorithm (see \cite[Theorem 2.6]{isaacson2012analysis} or \cite[Theorem 9.5]{higham2002accuracy} for details). For this reason, understanding the growth factor under different permutation strategies is of both theoretical and practical importance. Using exact arithmetic, the growth factor of Gaussian elimination is defined as
$$g(A):= \frac{\max_{i,j,k} |a_{i,j}^{(k)}|}{\max_{i,j} |a_{i,j}|}.$$
When performing Gaussian elimination in finite precision, say, using only numbers that can be represented in base $\beta$ with a length $t$ mantissa, the algorithm suffers from round-off error, and the growth factor in this setting may be larger than $g(A)$. However, as we will see in Section \ref{sec:float}, when $t = \omega(\log_\beta^{2} n)$, the maximum growth factors in exact and floating point arithmetic are nearly identical (up to a $1-o(1)$ multiplicative factor) under complete pivoting (see Theorem \ref{thm:float}). For this reason, we focus almost exclusively (save for Section \ref{sec:float}) on exact arithmetic. The most popular and well-studied methods for permuting a matrix in Gaussian elimination are partial pivoting (requiring $|a_{i,k}^{(k)}| \le |a_{k,k}^{(k)}|$), complete pivoting (requiring $|a_{i,j}^{(k)}| \le |a_{k,k}^{(k)}|$), and the slightly less well-known rook pivoting (requiring $|a_{i,k}^{(k)}|,|a_{k,j}^{(k)}|  \le |a_{k,k}^{(k)}|$). The growth factor for partial pivoting is well understood in the worst case, and so, in this work, we primarily focus on complete pivoting and, to some extent, rook pivoting as well.

Let $\textbf{GL}_n(\mathbb{C})$ be the set of $n \times n$ non-singular complex matrices. For simplicity, when considering a given pivoting strategy, we simply restrict ourselves to the set of matrices that satisfy the constraints of the pivoting procedure without requiring pivoting. In particular, we define
\begin{align*}
    \mathbf{PP}_n(S) &= \{ A \in \textbf{GL}_n(\mathbb{C}) \cap S^{n \times n} \, | \, |a_{i,k}^{(k)}| \le |a_{k,k}^{(k)}| \text{ for all } i \ge k \}, \\
    \mathbf{CP}_n(S) &= \{ A \in \textbf{GL}_n(\mathbb{C}) \cap S^{n \times n} \, | \, |a_{i,j}^{(k)}| \le |a_{k,k}^{(k)}| \text{ for all } i,j \ge k \}, \\
    \mathbf{RP}_n(S) &= \{ A \in \textbf{GL}_n(\mathbb{C}) \cap S^{n \times n} \, | \, |a_{i,k}^{(k)}|,|a_{k,j}^{(k)}| \le |a_{k,k}^{(k)}| \text{ for all } i,j \ge k\},
\end{align*}
where $S$ is some arbitrary subset of $\mathbb{C}$ (typically $\mathbb{R}$ or $\mathbb{C}$). For instance, when performing Gaussian elimination with complete pivoting on a matrix in $\mathbf{CP}_n(S)$, no pivoting is required. Furthermore, when performing Gaussian elimination with complete pivoting on a matrix in $\textbf{GL}_n(\mathbb{C}) \cap S^{n \times n}$, the resulting permuted matrix is in $\mathbf{CP}_n(S)$. We also stress that the role of the set $S$ is to constrain only the entries of the input matrix $A$; the sub-matrices $A^{(k)}$, $k >1$, during Gaussian elimination need not have entries in $S$ (though if $S = \mathbb{R}$, this will of course be the case). We denote the supremum of the growth factor for a set $\mathbf{X} \subset \mathbb{C}^{n \times n}$ by $g\big[\mathbf{X}\big]$, e.g., $g\big[\mathbf{CP}_n(\{0,1\})\big]$ is the maximum growth factor of a non-singular $n \times n$ binary matrix under complete pivoting. For all sets $\mathbf{X}$ under consideration in this work, this supremum is a maximum. In figures and tables, we use $g_n$ as shorthand for $g\big[\mathbf{CP}_n(\mathbb{R})\big]$.

\subsection{Related Work}
\label{subs:relatedtech}
The maximum growth factor for partial pivoting is well understood. This quantity is known to be exactly $2^{n-1}$ for $n \times n$ complex matrices,
achieved by Wilkinson's famous example matrix  \cite[p.212]{Wilkinson1965AEP} (see \cite{higham1989large} for all such real matrices).
For complete pivoting, much less is known. A classical result, due to Wilkinson, bounds the growth factor using only Hadamard's inequality \cite[Equation 4.15]{wilkinson1961error}, and produces the estimate
\begin{equation}\label{eqn:wilk}
g\big[\mathbf{CP}_n(\mathbb{C})\big] \le \sqrt{n} \big( 2 \; 3^{1/2} \; ... \; n^{1/(n-1)} \big)^{1/2} \le 2 \sqrt{n} \, n^{\ln (n)/4}.
\end{equation}
Minor improvements to this estimate are possible using the inexactness of Hadamard's inequality, but to date no non-trivial improvement (say, in the exponential constant) is known, even when restricted to real numbers. This estimate has historically been considered quite pessimistic; it was thought that the growth factor for real matrices under complete pivoting is at most $n$: 
\begin{conjecture}[Folklore?\footnote{See discussion in Section 1.1 regarding attribution.}]\label{conj:cryer}
$g\big[\mathbf{CP}_n(\mathbb{R})\big] \le n$,
with equality achieved only by Hadamard matrices.
\end{conjecture}
The complex analogue of this conjecture is clearly not true, as illustrated by the dimension three example \cite{tornheim1965maximum,tornheim1970maximum}
$$A = \begin{pmatrix}1 & 1 & 1 \\ 1 & z & z^{-1} \\ 1 & z^{-1} & z \end{pmatrix},$$
which, when $z = \big(-1 + 2 \sqrt{2} i\big)/3$, has growth factor $16/(3 \sqrt{3}) \approx 3.07$. As noted by Higham, Conjecture \ref{conj:cryer} was one of the most famous conjectures in numerical analysis \cite{higham2002accuracy}. Attempting to bound or numerically compute the growth factor for small values of $n$ was a reasonably active area of research. For instance, there are a number of proofs that the maximum third and fourth pivots are $2.25$ and $4$, respectively (see Cryer \cite{cryer1968pivot}, Tornheim \cite{tornheim1964maximum, tornheim1965maximum,tornheim1969maximum, tornheim1970maximum}, Cohen \cite{cohen1974note}, and Day and Peterson \cite{day1988growth}). Tornheim also showed that the maximum fifth pivot is bounded above by $4 \,\tfrac{17}{18}$ \cite{tornheim1969maximum,tornheim1970maximum}. Conjecture \ref{conj:cryer} was eventually shown to be false in dimension $13$ by Gould in IEEE double precision floating point arithmetic \cite{gould1991growth}, and soon after by Edelman in exact arithmetic \cite{edelman1992complete}. Since these results, very little progress has been made on the asymptotic behavior of the maximum growth factor under complete pivoting or the exact values of growth for small choices of $n$.

Rook pivoting is relatively understudied compared to partial and complete pivoting, despite, in some sense, containing the best characteristics of both methods. In practice, the expected number of comparisons required should be roughly the same order of computation as partial pivoting, see \cite{foster1997growth,poole2000rook} for empirical results and theorems of this type for certain restrictive classes of random matrices. In addition, rook pivoting has a quasi-polynomial upper bound on the maximal growth factor of 
\begin{equation}\label{eqn:rook_bound}
   g\big[\mathbf{RP}_n(\mathbb{C})\big] \le \frac{3}{2} \, n^{3\ln(n)/4},
\end{equation}
as shown by Foster \cite{foster1997growth}. Similar to complete pivoting, the gap between worst-case constructions and upper bounds is quite large.

The growth factor has also been studied in a variety of other contexts. Trefethen and Schreiber studied the average growth factor over some distributions and numerical observed that for complete pivoting the growth factor appeared to exhibit a $n^{1/2}$ type behavior \cite{trefethen1990average}. Higham and Higham have given numerous examples of matrices from practical applications with order $n$ growth factor \cite{higham1989large}, and recently produced a class of random matrices with growth of order $n/ \log n$ \cite{higham2021random} (both for any pivoting strategy). Sankar, Spielman, and Teng provided a smoothed analysis of growth factor without pivoting, proving that if a matrix is perturbed, it unlikely to have large growth factor \cite{sankar2006smoothed} (in Sankar's thesis, the more complicated case of partial pivoting was also considered \cite{sankar2004smoothed}). Recently Huang and Tikhomirov obtained new results exploring the average case analysis for partial pivoting \cite{huang2022average}. Parker proved that, using random butterfly matrices, any non-singular matrix can be transformed into one that does not require pivoting \cite{parker1995random}; Peca-Medlin and Trogdon further analyzed the benefits of butterfly matrices for a variety of pivoting strategies in \cite{peca2022growth}. Townsend produced bounds for the growth factor when non-optimal pivots are used \cite{townsend2016gaussian}.

\subsection{ Contributions of this paper}
In this work, we prove a number of results regarding the maximum growth factor under complete pivoting, strengthen various conjectures, 
provide strong evidence for some results, and perform extensive numerical computations.

\begin{table}[t]
\caption {\label{tab:goulddata} {In 1991, Gould \cite{gould1991growth} presented a Table 3.1:
Maximum Growth Factors Encountered.  We thought it would be of interest
to present the maximum growth factors we encountered  over 30 years later side by side.
The {\bf \color{blue} blue} number 13.0205 is Gould's 1991 first surprising example
of a matrix with $g(A)>n$.  The {\bf \color{red} red} numbers show that
it is possible to find examples even when $n=11$ and $12$. The {\bf \color{red} red} and {\bf \color{magenta} magenta} numbers
are improvements over previously computed results. Only the bold face {\bf   black} numbers are known to equal $g_n$ exactly.
}}

\begin{tabular}{cc}
\begin{tabular}{|r|l|c|l|l}
\multicolumn{3}{c}{$g_n$ Known Exactly} \\
\cline{1-3}
1 & \multicolumn{ 2}{c|}{\bf  1}   \\
2 & \multicolumn{2}{c|} {\bf  2}   \\
3 &\multicolumn{2}{c|}   {\bf 2.25}  \\
4 &  \multicolumn{2}{c|}{\bf  4}   \\
\cline{1-3}
\multicolumn{3}{c}{}
\\ 
\multicolumn{3}{c}{Ours $=$ same  as \cite{day1988growth,gould1991growth}} \\
\cline{1-3}
5 & \multicolumn{2}{c|}{4.1325}   \\
6 & \multicolumn{2}{c|}{5}    \\
8 & \multicolumn{2}{c|}{8}   \\
\cline{1-3}
\multicolumn{3}{c}{}
\\
\multicolumn{3}{c}{\hspace{.3in} Ours /  Gould \cite{gould1991private}} \\
\cline{1-3}
18 & {\bf \color{magenta} 21.25}  &  20.45 \\
20 & {\bf \color{magenta} 24.71}   & 24.25 \\
25 & {\bf \color{magenta} 33.67}   & 32.99 \\ \cline{1-3}
\multicolumn{3}{c}{}
\\ 
\end{tabular} 
 & 
 \begin{tabular}{|r|l|c|l|l} \\[-.6in]
 \multicolumn{3}{c}{\hspace{.3in} Ours /  Gould  \cite{gould1991growth}} \\
\cline{1-3} 
7 &  6.05 & 6 \\
9 & 8.69 & 8.4305 \\
10 &  9.96 & 9.5294 \\
{\bf \color{red}  11} &  {\bf  \color{red} 11.05 } & 10.4627 \\ 
{\bf \color{red} 12} &  {\bf  \color{red} 12.55 } & 12 \\
13 &  {\bf \color{magenta}  13.76}  &  {\bf \color{blue} 13.0205} \\  
14 &   {\bf \color{magenta} 15.25} & 14.5949 \\
15 &  {\bf \color{magenta}  16.92 } &  16.1078  \\
16 &   {\bf \color{magenta}  18.46 } & 18.0596 \\ 
\cline{1-3}
\multicolumn{3}{c} {\bf \color{red} Conjecture: $g_n
>n $ iff $n >10$} \\
\multicolumn{3}{c}{ \bf \color{blue} 13.0205 = 1991 surprise} 
\\
\multicolumn{3}{c}{}
\\[0in]
\multicolumn{3}{c}{$\! \! \!  \!  \!  \!  \!  \hspace*{-.6in}  \longleftarrow$ as documented in \cite{edelman1992complete} }
\end{tabular}
\end{tabular}

\end{table}

Through numerical search, and stability and extrapolation results, we provide improved lower bounds for the maximum growth factor:

\begin{theorem}\label{thm:cp_growth}
$g\big[\mathbf{CP}_n(\mathbb{R})\big] \ge 1.0045\, n$ for all $n > 10$, and $\limsup_n \big(g\big[\mathbf{CP}_n(\mathbb{R})\big] /n\big) \ge 3.317$.
\end{theorem}

This is the first proof that Conjecture \ref{conj:cryer} is false for all $n >10$, and also the first proof that illustrates a multiplicative gap away from $n$.

In addition, we also provide asymptotic lower bounds for rook pivoting. By noting that the set of rook pivoted matrices are closed under Kronecker products, we convert finite results into lower bounds for the exponent of the growth factor, showing that rook pivoting can exhibit super-linear growth:

\begin{theorem}\label{thm:rp_growth}
$g\big[\mathbf{RP}_{n}(\mathbb{R})\big] > \frac{1}{641}n^{1.669}$ for all $n \in \mathbb{N}$.
\end{theorem}

Numerical search is a key ingredient in the proofs of both Theorems \ref{thm:cp_growth} and \ref{thm:rp_growth}, and our numerical results also provide insights beyond the aforementioned theorems, which we briefly summarize through the following figures and tables:
\begin{itemize}
\item Table \ref{tab:goulddata} shows improvements compared to previously known data.
\item Table \ref{tab:Hadamard} outlines the implications of our results for low order Hadamard matrices.
\item Table \ref{tab:cp_r} tabulates our numerical results for every $n=1:75$ and also $n=100$.
\item Figure \ref{fig:alln} plots the numerical values from Table 3.
\end{itemize}
The reported numerical computations were performed 
in Julia using the modern JuMP (Julia for Mathematical Programming) \cite{DunningHuchetteLubin2017} package. We note that, when $n = 52$ we have found a matrix for which  the growth factor is greater than $2n$, and at $n=100$ the growth factor is well above $3n$. We also found a matrix for which the growth factor with rook pivoting is $640$ at $n = 48$. We discuss our methodology for the computation of these results and state a pair of natural conjectures in Subsection \ref{sub:max_growth}.

\begin{table}[t]
\caption {\label{tab:Hadamard}  Hadamard matrices: For decades, Hadamard matrices, interesting in their own right,
seemed relevant to the growth factor problem.  Gould \cite{gould1991growth} shattered
that notion with his computation for $n=16$.  We observed that the notion can already
be shattered partially at $n=8$ and fully at $n=12$:}
\begin{tabular}{ll}
$n=4$: & Mathematics shows $g_4=4$ and the optimum is Hadamard \\
$n=8$: &  $g_8=8$ remains a conjecture, however one new observation \\
& is that the matrix need not be Hadamard. \\
$n=12$: & We are the first to report a 12x12 matrix with $g_{12}>12$, thus\\
&  showing Hadamard matrices do not maximize growth for $n = 12$. \\
$n=16$: & Gould reported the discovery of a 16x16 matrix with $g_{16}>16$,  \\
&  therefore Hadamard matrices do not maximize growth for $n = 16$. \\ & We observed a slightly more optimal matrix. 
\end{tabular}
\end{table}

\begin{table}[t]
\caption {GECP Data computed by JuMP for matrices of dimensions $n=1:75$ and $100$ in exact arithmetic.} \label{tab:cp_r}
\begin{tabular}{| l | r |  @{\hskip .2in}   | l | r | @{\hskip .2in} | l | r |@{\hskip .2in} | l | r |@{\hskip .2in} | l | r |}  \hline
$n=$ & $g \ge \ \downarrow   $ & $n=$ &  $g \ge \ \downarrow   $ & $n=$ &  $g \ge \ \downarrow   $ & $n=$ &  $g \ge \ \downarrow   $& $n=$ &  $g \ge \ \downarrow   $ \\  \hline
 $\;1$ &  $1$ & $16$ & $18.46$ & $31$  & $45.43$ & $46$ & $85.85$&  $61$  & $137.55$ \\
 $\;2$ & $2$ & $17$ & $19.86$  & $32$ & $47.74$ & $47$ & $87.54$&  $62$  & $141.83$  \\
 $\;3$ & $9/4$ & $18$ & $21.25$ &$33$ & $50.36$  & $48$ & $91.44$& $63$ & $144.72$ \\
 $\;4$ & $4$ & $19$ & $22.85$ & $34$ & $52.78$  & $49$ & $94.72$ &$64$  &$148.05$\\
 $\;5$ & $4.13$ & $20$ & $24.71$ & $35$ & $54.84$  & $50$ & $97.24$& $65$ & $153.98$  \\ \hline
 $\;6$ & $5$ & $21$ & $26.21$ & $36$ & $57.66$&  $51$  & $101.82$ & $66$ & $157.05$ \\
 $\;7$ & $6.05$ & $22$ & $28.01$ & $37$ & $59.91$& $52$  & $104.61$   & $67$ & $162.20$\\
 $\;8$ & $8$ &$23$ & $29.72$& $38$ & $63.18$ &  $53$ & $108.09$  & $68$ & $166.89$ \\
 $\;9$ & $8.69$ & $24$ & $31.63$&  $39$ & $64.87$ & $54$ &$111.19$  & $69$ & $171.33$ \\
 $10$ & $9.96$ &$25$ & $33.67$& $40$ & $67.52$ & $55$ & $114.76$  & $70$ & $174.45$\\ \hline
$11$ & $11.05$ &  $26$ & $34.96$ & $41$  & $70.44$& $56$ & $118.18$& $71$  & $182.98$ \\
$12$ & $12.55$ & $27$ & $36.88$& $42$ & $73.49$& $57$ & $121.90$  & $72$  & $184.91$ \\
$13$ & $13.76$ & $28$ & $39.05$& $43$ & $77.68$ & $58$ & $126.23$  & $73$ & $190.57$\\
$14$ & $15.25$ & $29$ & $41.46$ & $44$ &$79.25$ & $59$ & $129.42$   & $74$  &$193.28$\\
$15$ & $16.92$ & $30$ & $43.40$ & $45$ & $82.56$ & $60$ & $134.27$  & $75$ & $196.79$  \\  \hline
\end{tabular}
\vspace{-.1in}
$$\hspace*{4.6in} \vdots \vspace{-.07in}$$

\hspace*{4.6in} \begin{tabular}{|l|r|}
\hline
$100 \;$ & $331.71$ \\ \hline
\end{tabular}
\end{table}

We also outline our more theoretical results:

We show that the maximum growth factor over matrices with entries restricted to a subset of $\mathbb{R}$ is nearly equal to the maximum growth factor over all real matrices

\begin{theorem}[Simplified Version of Theorem \ref{thm:restrict}]\label{thm:inform_restrict}
For any $S \subset \mathbb{R}$, $g\big[\mathbf{CP}_{14n^2}(S)\big] \ge \big(\emph{diam}(S)/[2\max (S)] \big)\, g\big[\mathbf{CP}_{n}(\mathbb{R})\big]$ for all $n \in \mathbb{N}$.
\end{theorem}

\definecolor{myyellow}{rgb}{1, .8, 0}

\begin{figure}[t]
    \centering
    \subfigure[]{\includegraphics[width=.49\linewidth]{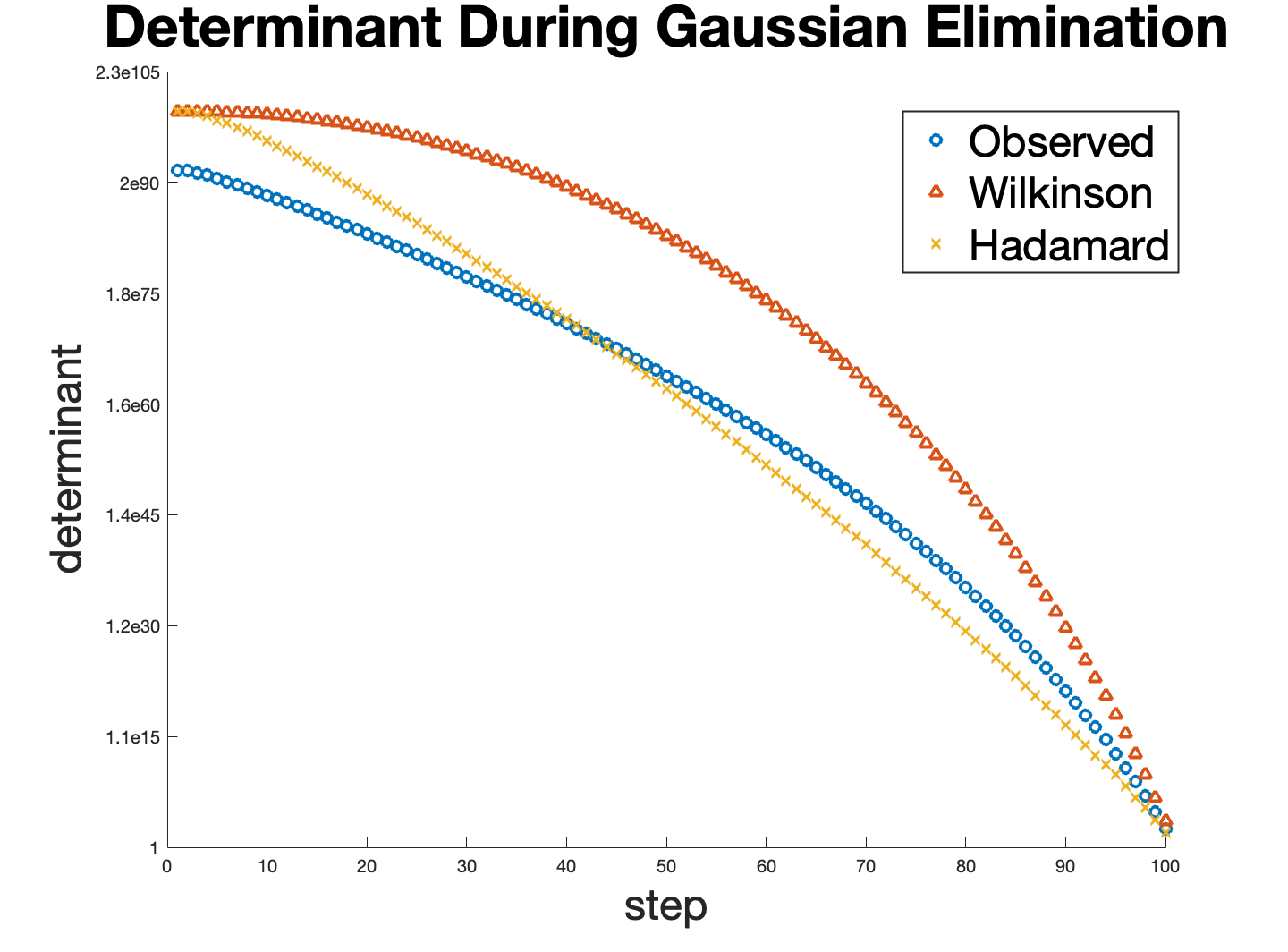}} 
    \subfigure[]{\includegraphics[width=.49\linewidth]{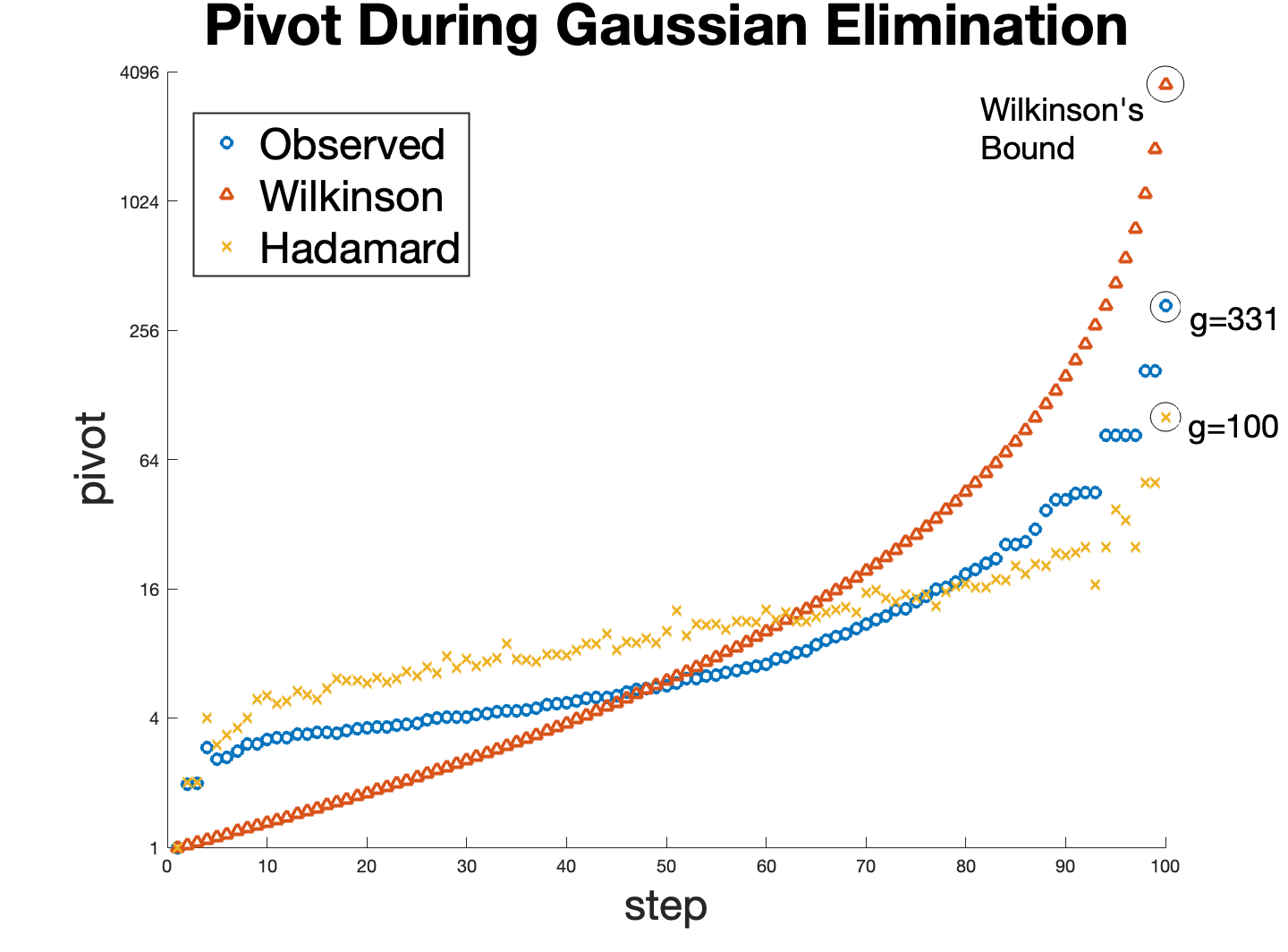}} 
    \caption{
We compare the (modulus of the) determinant and pivots of $A^{(k)}$, $k = 1,...,n$, under GECP for three examples of size $n=100$:
{\bf \color{red} Red: Wilkinson's bound}; {\bf \color{myyellow} Yellow: a particular $n=100$ Hadamard matrix};
{\bf \color {blue} Blue: our observed maximum matrix}.  
(a)  reveals that at least on an admittedly muted log scale, the observed determinant curve qualitatively
is bending in a manner resembling Wilkinson's bound, while the Hadamard data feels qualitatively different, and thus, less relevant.
(b)  suggests the same conclusions  as those of (a) and  also suggests that ``slow and steady wins the race" rather
than ``greedy."
    }
     \label{fig:growth}
\end{figure}

This implies that understanding the growth factor for any restricted set, say, binary matrices, is equivalent (up to polynomial factors) to understanding growth for all real matrices (i.e., if growth for binary matrices is polynomial, then it is polynomial for all matrices, and if growth for all matrices is super-polynomial, then it also is for binary matrices).
We note that the $O(n^2)$ relationship is certainly 
pessimistic for many sets $S$ of interest; our purpose here 
is merely to show it is possible within small polynomial factors to
find such sweeping results.

We also show that the growth factors  under floating point arithmetic and exact arithmetic are nearly identical.

\begin{theorem}[Simplified Version of Theorem \ref{thm:float}]\label{thm:informal_float}
Let
$$t \ge 1 +  \log_\beta\big[ 5 n^3 g^2\big[\mathbf{CP}_n(\mathbb{R})\big] \big].$$
Then the maximum growth factor for a real $n \times n$ matrix under floating point arithmetic with base $\beta$ and mantissa length $t$ is at most $(1+1/n) \,  g\big[\mathbf{CP}_n(\mathbb{R})\big]$.
\end{theorem}

 Theorem \ref{thm:float}
 treats a longstanding gap in the numerical analysis literature, a field where so much energy is  devoted to the
distinction between floating point and exact computations, but  in the context of growth factors, this has not been 
analyzed to date. This theorem provides a link between rounding error analyses that allow the unit round-off to
tend to zero and those that give even more pessimistic bounds because of this difference (see, for example, \cite[Theorem 2.4]{isaacson2012analysis} and the discussion preceding it).

\subsection{Maximum Growth Factors Encountered}\label{sub:max_growth}

With modern software and architecture we were able to find growth factors for matrices well beyond $n=25$ as found by Gould \cite{gould1991growth} 
and also we were able to find larger growth for matrices as small as $n=7$.   No doubt future researchers will be able to 
improve our results in the same manner.

As the optimal growth problem is a constrained optimization problem it is natural
to run optimization software.
In 1988, Day and Peterson  \cite{day1988growth} posed the problem as a function
of the $n^2$ elements of the matrix and reported some success with
 the FORTRAN77 nonlinear programming package NPSOL \cite{NPSOL}.
 By contrast, Gould considered the advantages of posing the problem as an optimization
 over $n^2 + (n-1)^2 + \ldots +1 $ variables with constraints.  He used the
 FORTRAN77 LANCELOT package that he codeveloped \cite{LANCELOT}.
 (LANCELOT is an acronym for ``Large And Nonlinear Constrained Extended Lagrangian Optimization Techniques).

\begin{figure}[t!]
    \centering
    \includegraphics[width=.6\linewidth]{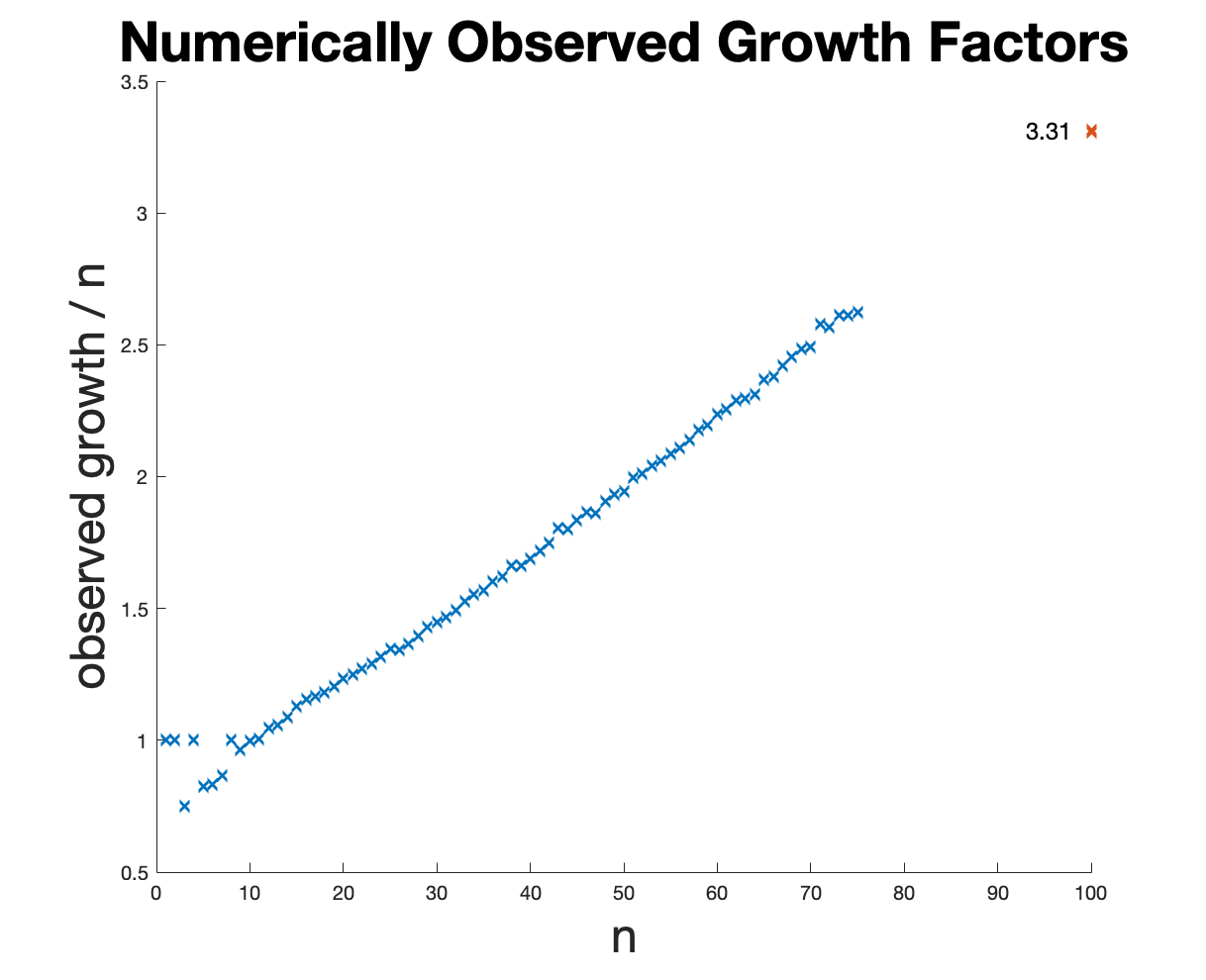}
    \caption{The ratio between numerically observed growth factors and matrix size for $n$ equals $1$ to $75$ and $100$. Only the values for sizes $n = 1,2,3,4$ are known mathematically to be the exact maximal growth factor though we suspect at least for the smaller values of $n$ we are achieving the maximum with
    our JuMP software. This data leads us to make Conjecture \ref{conj:superlinear}. \label{fig:alln}
    }
     \label{fig:growthratio}
\end{figure}

 We chose to follow Gould's approach but chose to use the modern JuMP (acronym: Julia for Mathematical Programming) \cite{DunningHuchetteLubin2017} software library
 with IPOPT (acronym: Interior Point Optimizer) to formulate and solve the minimization problem. The software advantage of using Julia is that the problem can be naturally formulated
 in a manner very similar to the mathematics \cite{bezanson2017julia}. The optimization engine was the COIN-OR Ipopt (interior point optimizer) package called through Ipopt.jl. The software and results may be found in the \href{https://github.com/alanedelman/CompletePivotingGrowth}{online repository} \cite{ourrepo}.

JuMP was run in parallel with 64 randomly chosen starting points on 64 separate threads and the winner, the largest growth factor, was saved.
Computations were performed on a server located at MIT consisting of two AMD EPYC 7502 32-Core Hyperthreading Processors, typically using 64 of the 128 hyperthreads at a time so that others could use the machine for their own work.

 We studied $n=10$ extensively numerically and never exceeded 9.96, so we feel the evidence is very strong to state the following conjecture:
 \begin{conjecture}
The growth factor for complete pivoting: $g\big[\mathbf{CP}_n(\mathbb{R})\big]  \ge n $ if and only if $n > 10$.
\end{conjecture}

In addition, though exact asymptotic estimates for growth factor remain elusive, we feel that we have seen sufficient numerical evidence (see Figure \ref{fig:growthratio}) to conjecture that the growth factor is super-linear (recall, $f(n) = \omega(g(n))$ if $\lim_{n\rightarrow\infty} f(n)/g(n) = \infty$):

 \begin{conjecture}\label{conj:superlinear}
$g\big[\mathbf{CP}_n(\mathbb{R})\big]  = \omega( n)$.
\end{conjecture}

\subsection{Remainder of Paper}
The remainder of the paper is organized as follows. In Section \ref{sec:stable}, we prove a stability lemma, which, when combined with numerical experiments and extrapolation results, imply lower bounds for both complete and rook pivoting. In Section \ref{sec:constrain}, we show that the maximum growth factor for matrices with entries in an arbitrary non-trivial set $S$ is nearly as large as the maximum growth factor over all real matrices. In Section \ref{sec:float}, we consider the growth factor in finite precision, and show that only polylogarithmically many bits (in $n$) are needed for this quantity to be at most a constant times the growth factor in exact arithmetic. In Section \ref{sec:numerics}, we describe the numerical programs used to search for large growth factors, prove extrapolation results, and report our mathematically verified results. Finally, in Section \ref{sec:rook}, we study the growth factor for rook pivoting.

\section{Key Stability Lemma: Almost Completely Pivoted is Almost Completely Pivoted}\label{sec:stable}

The ``stability lemma" in this  section  is a critical technical ingredient in the majority of the theorems that follow.  One immediate application follows 
a longstanding tradition of numerical analysis : backward error analysis.
The lemma shows that if a numerical computation, such as the computations described in Sections \ref{sec:intro} and \ref{sec:numerics}, provides a computed growth factor for a ``nearly" completely pivoted matrix, then there is 
a ``nearby" matrix which has a ``nearby" growth factor for complete pivoting.

For a given $\bm{\varepsilon} = (\varepsilon_1,...,\varepsilon_{n-1}) \in \mathbb{R}^{n-1}$, $\varepsilon_i>-1$ for $i =1,...,n-1$, we define
$$\mathbf{CP}^{\bm{\varepsilon}}_n(S) = \{ A \in \textbf{GL}_n(\mathbb{C}) \cap S^{n \times n} \, | \, |a_{i,j}^{(k)}| \le (1+\varepsilon_k) |a_{k,k}^{(k)}| \text{ for all } i,j \ge k, \, (i,j) \ne (k,k) \},$$
where $S$ is some arbitrary subset of $\mathbb{C}$ (typically $\mathbb{R}$ or $\mathbb{C}$), e.g., the set of matrices that are ``almost" completely pivoted (or, for $\varepsilon_k<0$, ``overly" completely pivoted) up to a multiplicative error of $\varepsilon_k$ at the $k^{th}$ step of Gaussian elimination. When $\varepsilon_1=...=\varepsilon_{n-1}>0$, these sets are generally referred to as threshold-pivoted matrices.

We first prove the following lemma, showing that every matrix in $\mathbf{CP}^{\bm{\varepsilon}}_n(S)$ is close to a matrix in $\mathbf{CP}^{\bm{\delta}}_n(S)$, where $\varepsilon_i \ge 0 \ge \delta_i$ for all $i = 1,...,n-1$. The statement and proof for the more general case in which $\varepsilon_i$ and $\delta_i$ may have the same sign is similar, but is slightly more complicated and not needed for our purposes.  We also give an algorithmic description of the procedure in the proof of Lemma \ref{lm:cp_eps_stable} for the case $\bm{\delta}=0$ in Algorithm \ref{alg:eps_stable}, as this subroutine is a crucial part of converting numerically computed results to mathematical proofs of lower bounds.

\begin{algorithm}[t]
\caption{ Practical Algorithm to turn optimization output  to theory: \newline
Convert $A \in \textbf{GL}_n(\mathbb{\mathbb{R}})$ to $B \in \mathbf{CP}_n(\mathbb{R})$}\label{alg:eps_stable}
\begin{algorithmic}[1]
\State $B^{(1)} \gets \textbf{Rational}(A)$ \Comment{Convert floating point matrix to rational matrix}
\For{$k = 1:n-1$}
\State $B^{(k+1)} \gets B^{(k)}- B^{(k)}_{:,k} B^{(k)}_{k,:}/B^{(k)}_{k,k}$
\EndFor
\For{$k = (n-1):-1:1$}
\State $\delta \gets \max \big\{ \max_{i >k} \big(B^{(k)}_{i,k}/B^{(k)}_{k,k}\big)^2 , \, \max_{j >k} \big(B^{(k)}_{k,j}/B^{(k)}_{k,k}\big)^2 , \, \max_{i,j >k} \big|B^{(k)}_{i,j}/B^{(k)}_{k,k}\big| \big\}$
\If{$\delta >1$}  \Comment{else matrix is already CP \hspace*{1.5in} }
    \State $\delta_h \gets  \textbf{Rational}\big( \textbf{NextFloat}(\delta^{1/2})\big)$ \Comment{Approx. $\delta^{1/2}$, round up, make rational}
    \State $B^{(k)}_{:,k} \gets \delta_h B^{(k)}_{:,k}$, $B^{(k)}_{k,:} \gets \delta_h B^{(k)}_{k,:}$
    \If{$k>1$}
    \For{$ i = 1:(k-1)$}
    \State $B^{(i)}_{k,k} \gets B^{(i)}_{k,k} + (\delta_{h}^2-1) B^{(k)}_{k,k}$
    \State $B^{(i)}_{k,k+1:n} \gets B^{(i)}_{k,k+1:n} + (\delta_{h}-1)B^{(k)}_{k,k+1:n}$
    \State $B^{(i)}_{k+1:n,k} \gets B^{(i)}_{k+1:n,k} +  (\delta_{h}-1)B^{(k)}_{k+1:n,k}$
    \EndFor
    \EndIf
    \EndIf
    \EndFor
    \State $B \gets B^{(1)}$
\end{algorithmic}
\end{algorithm}


\begin{lemma}\label{lm:cp_eps_stable}
Let $\bm{\varepsilon} = (\varepsilon_1,...,\varepsilon_{n-1})$ and $\bm{\delta} = (\delta_1,...,\delta_{n-1})$ satisfy  $-1< \delta_i \le 0 \le \varepsilon_i$ for $i =1,...,n-1$. Then, for every $A \in \mathbf{CP}^{\bm{\varepsilon}}_n(S)$, where $S$ equals $\mathbb{R}$ or $\mathbb{C}$, there exists a matrix $B \in \mathbf{CP}^{\bm{\delta}}_n(S)$ such that $b_{n,n}^{(k)} = a_{n,n}^{(k)}$ for all $k =1,...,n$, and
$$\big|b_{i,j}^{(k)}-a_{i,j}^{(k)}\big| \le \max_{\min\{i,j\} \le \ell \le n-1}  \frac{\displaystyle{\bigg[\bigg(\frac{1+\varepsilon_\ell}{1+\delta_\ell}\bigg)^2 -1\bigg]\big|a_{\ell,\ell}^{(\ell)}\big|}}{\displaystyle{ \prod_{p = \min\{i,j\}}^{\ell-1} 1+\delta_p}}  +\sum_{m = \min\{i,j\}}^{\ell-1} \frac{\displaystyle{(\varepsilon_{m} -\delta_m)\big|a_{m,m}^{(m)}\big| }}{\displaystyle{ \prod_{p = \min\{i,j\}}^{m} 1+\delta_p }}.$$

\end{lemma}

\begin{proof}
To construct $B \in \mathbf{CP}^{\bm{\delta}}_n(S)$, we iteratively define the entries $b_{i,j}^{(k)}$, starting with $k = n$ and working backwards from $k =n$ to $k =1$. The key to this construction is that we scale the row and column of the pivot entry of each matrix $A^{(k)}$ by a fixed multiplicative factor. This operation leaves entries $a_{i,j}^{(\ell)}$, $i,j>k$, unchanged, and so during our procedure each entry is changed at most once. The factor depends on both the maximum magnitude entry $|a_{i,j}^{(k)}|$ over all $i,j>k$ and the maximum over $i=k,j>k$ and $j=k, i>k$. This allows error to propagate additively rather than multiplicatively.

Let $B^{(n)} := A^{(n)}$ and
$$B^{(k)} := \begin{pmatrix} (1+\gamma_k) \, a^{(k)}_{k,k} & \sqrt{1+\gamma_k} \, A^{(k)}_{k,k+1:n} \\
 \sqrt{1+\gamma_k} \, A^{(k)}_{k+1:n,k}  & A^{(k)}_{k+1:n,k+1:n} + B^{(k+1)} - A^{(k+1)}
\end{pmatrix}$$
for $k = 1,...,n-1$, where $\gamma_n:=0$ and
$$ \gamma_k := \max\bigg\{\bigg(\frac{1+\varepsilon_k}{1+\delta_k}\bigg)^2 -1, \frac{\varepsilon_k - \delta_k + \max_{i>k} \gamma_i |a_{i,i}^{(i)}|/|a_{k,k}^{(k)}|}{1+\delta_k} \bigg\}.$$
The quantity $\gamma_k |a_{k,k}^{(k)}|$ is monotonically decreasing with $k$ (as $\delta_k\le 0$), and so we may equivalently write
$$ \gamma_k |a_{k,k}^{(k)}| = \max \bigg\{ \bigg[\bigg(\frac{1+\varepsilon_k}{1+\delta_k}\bigg)^2 -1\bigg]|a_{k,k}^{(k)}|, \frac{(\varepsilon_k - \delta_k ) |a_{k,k}^{(k)}| +  \gamma_{k+1} |a_{k+1,k+1}^{(k)}|}{1+\delta_k} \bigg\},$$
or
$$\gamma_k |a_{k,k}^{(k)}| = \max_{\ell \ge k}  \bigg[\bigg(\frac{1+\varepsilon_\ell}{1+\delta_\ell}\bigg)^2 -1\bigg]|a_{\ell,\ell}^{(\ell)}| \, \prod_{p = k}^{\ell-1} \frac{1}{1+\delta_p}  +\sum_{m = k}^{\ell-1} (\varepsilon_{m} -\delta_m)|a_{m,m}^{(m)}| \, \prod_{p = k}^{m} \frac{1}{1+\delta_p }.$$
Our definitions of $B^{(k)}$ are consistent with one another, as
$$B^{(k)}_{k+1:n,k+1:n} - \frac{B^{(k)}_{k+1:n,k} B^{(k)}_{k,k+1:n}}{b^{(k)}_{k,k}}=A^{(k)}_{k+1:n,k+1:n} + B^{(k+1)} - A^{(k+1)} - \frac{  A^{(k)}_{k+1:n,k} A^{(k)}_{k,k+1:n}}{a^{(k)}_{k,k}} = B^{(k+1)}.$$
Furthermore, as $\varepsilon_k \le \sqrt{1+\gamma_k}$, for $i\le j$ ($j>i$ is similar),
\begin{align*}
    |b_{i,j}^{(k)}-a_{i,j}^{(k)}| = |b_{i,j}^{(i)}-a_{i,j}^{(i)}| &\le \max \{\gamma_i |a_{i,i}^{(i)}|, (\sqrt{1+\gamma_i}-1) |a_{i,j}^{(i)}|\} \\  &\le \max \{ \gamma_i, (1+\varepsilon_i)(\sqrt{1+\gamma_i}-1) \} |a_{i,i}^{(i)}| \\
    &= \gamma_i |a_{i,i}^{(i)}|.
\end{align*}
What remains is to verify that $B \in \mathbf{CP}^{\delta}_n(S)$. We proceed by induction from $k = n-1$ to $k =1$. We need only consider entries in the lower right block of $B^{(k)}$, as, for $i>k$,
$$ |b_{i,k}^{(k)}| = \sqrt{1+\gamma_k} |a_{i,k}^{(k)}| \le (1+\varepsilon_k) \sqrt{1+\gamma_k} |a_{k,k}^{(k)}| = \frac{1+\varepsilon_k}{\sqrt{1+\gamma_k}} |b_{k,k}^{(k)}| \le (1+\delta_k) |b_{k,k}^{(k)}|,$$
and the same bound holds for $b_{k,j}^{(k)}$, $j > k$.

When $k = n-1$,
$$|b_{n,n}^{(n-1)}| = |a_{n,n}^{(n-1)}| \le (1+\varepsilon_{n-1}) |a_{n-1,n-1}^{(n-1)}| = \frac{1+\varepsilon_{n-1}}{1+ \gamma_{n-1}} |b_{n-1,n-1}^{(n-1)}| \le (1+ \delta_{n-1})|b_{n-1,n-1}^{(n-1)}|.$$
Suppose the statement holds for $k = \ell+1,...,n-1$, $\ell<n-1$, and consider $b_{i,j}^{(\ell)}$, $\ell < i \le j$ ($\ell < j < i$ is similar). We have
\begin{align*}
    |b_{i,j}^{(\ell)}| 
    &\le | a_{i,j}^{(\ell)}| + |b_{i,j}^{(i)} - a_{i,j}^{(i)}| \\ &\le (1+\varepsilon_\ell)| a_{\ell,\ell}^{(\ell)}| + \gamma_i |a_{i,i}^{(i)}| \\ &= |a_{\ell,\ell}^{(\ell)}| \big((1+\varepsilon_\ell) +\gamma_i |a_{i,i}^{(i)}|/|a_{\ell,\ell}^{(\ell)}| \big) \\
    &= (1+\delta_\ell) |b_{\ell,\ell}^{(\ell)}| \frac{(1+\varepsilon_\ell) +\gamma_i |a_{i,i}^{(i)}|/|a_{\ell,\ell}^{(\ell)}|}{(1+\delta_\ell)(1+\gamma_\ell)} \\ 
    &\le (1+\delta_\ell) |b_{\ell,\ell}^{(\ell)}|,
\end{align*}
and therefore $B \in \mathbf{CP}^{\delta}_n(S)$.
\end{proof}

A simpler, but weaker version of the above result, relating the maximum growth factor under threshold complete pivoting to that of complete pivoting, is as follows.

\begin{corollary}\label{cor:cp_stab}
Let $S$ equal $\mathbb{R}$ or $\mathbb{C}$, and  $\bm{\varepsilon}=(\epsilon,...,\epsilon)$,  $\epsilon>0$. Then
$$ g\big[\mathbf{CP}_n(S)\big] \ge \frac{ g\big[\mathbf{CP}^{\bm{\varepsilon}}_n(S)\big]}{1 + \epsilon(2+\epsilon) g\big[\mathbf{CP}^{\bm{\varepsilon}}_{n-1}(S)\big] + \epsilon\sum_{i=1}^{n-2} g\big[\mathbf{CP}^{\bm{\varepsilon}}_i(S)\big]}$$
for all $n \in \mathbb{N}$.
\end{corollary}
\begin{proof}
The result follows from choosing an $A \in \mathbf{CP}^{\bm{\varepsilon}}_n(S)$ that achieves the maximum growth factor and choosing $B \in \mathbf{CP}_n(S)$ from Lemma \ref{lm:cp_eps_stable} as a lower bound for $g\big[\mathbf{CP}_n(S)\big]$. By Lemma \ref{lm:cp_eps_stable}, we have
$$g(A) = \frac{|a_{n,n}^{(n)}|}{\max_{i,j} |a_{i,j}|} =  \frac{|b_{1,1}|}{\max_{i,j} |a_{i,j}|} \, g(B) \le  \frac{|b_{1,1}|}{ |a_{1,1}|} \, g(B) = (1+\gamma_1) \, g(B),$$
where
$$\gamma_1 \le \max_{\ell<n} \epsilon(2+\epsilon) \frac{|a^{(\ell)}_{\ell,\ell}|}{|a_{1,1}|} + \epsilon \sum_{m = 1}^{\ell-1} \frac{|a^{(m)}_{m,m}|}{|a_{1,1}|} \le  \epsilon(2+\epsilon) g\big[\mathbf{CP}^{\bm{\varepsilon}}_{n-1}(S)\big] + \epsilon \sum_{m = 1}^{n-2} g\big[\mathbf{CP}^{\bm{\varepsilon}}_m(S)\big].$$
\end{proof}

In addition to being a crucial ingredient for our results, Corollary \ref{cor:cp_stab} also has some historical significance. This result, and the associated algorithm illustrates a way to convert almost completely pivoted matrices into matrices that are completely pivoted, without losing much in the growth factor. This has key similarities to Edelman's exact arithmetic extension of Gould's finite precision counterexample to Conjecture \ref{conj:cryer}, and provides some answers to Edelman's perturbation question for growth factor \cite{edelman1992complete}.

\section{Growth Factor for Constrained Entries}\label{sec:constrain}

In this section, we study the maximum growth factor of matrices in $\mathbf{GL}_n(\mathbb{R})\cap S^{n \times n}$ when $S$ is a small set (e.g., $\{0,1\}$). In particular, we aim to show that the maximum growth factor for matrices with entries restricted to some subset $S \subset \mathbb{R}$ is nearly the same as the growth factor over $\mathbb{R}$, up to a quadratic factor in the input $n$. To do so, we proceed as follows: First, we show that the maximum growth factor for matrices at least some prescribed distance from the boundary is almost as large as the maximum growth factor over the entire set (Lemma \ref{lm:cp_eps_ball}). Combining this result with the stability lemma of the previous section (Lemma \ref{lm:cp_eps_stable}) produces a lower bound for the maximum growth factor of sets of matrices that cover $\mathbf{CP}_n(S)$ sufficiently well (Lemma \ref{lm:cp_eps_cover}). Finally, using this lower bound, we show that if our restricted set $S$ is non-trivial (i.e., $|S|>1$), then we can almost achieve the maximum growth factor, up to a quadratic factor in $n$.

 We begin by characterizing a subset of $\mathbf{CP}_n(S)$ that is stable under entry-wise perturbations of size at most $\varepsilon$, i.e., matrices $A$ such that $\{ B \in S^{n \times n} \, | \, |a_{i,j} - b_{i,j}| \le \varepsilon \} \subset \mathbf{CP}_n(S)$. We have the following lemma.

\begin{lemma}\label{lm:cp_eps_ball}
Let $A \in \mathbf{CP}_n(S)$, $S$ equal $\mathbb{R}$ or $\mathbb{C}$, and $ \varepsilon >0$. If $|a_{i,j}^{(k)}| \le |a_{k,k}^{(k)}| - 2 \times 4^{k-1} \varepsilon$ for all $i,j = k,...,n$ (except $i=j=k$), $k = 1,...,n-1$, then
$$\{B \in S^{n \times n} \, | \, |a_{i,j} - b_{i,j}|\le \varepsilon \} \subset \mathbf{CP}_n(S),$$
and $$g\big[\{B \in S^{n \times n} \, | \, |a_{i,j} - b_{i,j}|\le \varepsilon \} \big] \ge g(A) - \varepsilon(4^{n-1}+g(A))/|a_{1,1}|.$$
\end{lemma}

\begin{proof}
Let $B \in S^{n \times n}$ satisfy $b^{(1)}_{i,j} = a^{(1)}_{i,j} + \theta^{(1)}_{i,j}$, where $|\theta^{(1)}_{i,j}| \le \varepsilon$. Then
   $$ b_{i,j}^{(2)} = \bigg[ (a_{i,j}^{(1)} + \theta^{(1)}_{i,j}) - \frac{(a_{i,1}^{(1)} + \theta^{(1)}_{i,1})(a_{1,j}^{(1)} + \theta^{(1)}_{1,j})}{(a_{1,1}^{(1)} + \theta^{(1)}_{1,1})}\bigg]+ a_{i,j}^{(2)} - \bigg[a_{i,j}^{(1)}- \frac{a^{(1)}_{i,1} a^{(1)}_{1,j}}{a^{(1)}_{1,1}}\bigg] = a_{i,j}^{(2)} + \theta^{(2)}_{i,j},$$
where
$$\theta^{(2)}_{i,j}:= \theta_{i,j}^{(1)} + \theta_{1,1}^{(1)} \frac{a_{i,1}^{(1)} a_{1,j}^{(1)}}{a_{1,1}^{(1)}(a_{1,1}^{(1)} + \theta_{1,1}^{(1)})} - \frac{\theta_{i,1}^{(1)}a_{1,j}^{(1)} +\theta_{1,j}^{(1)}a_{i,1}^{(1)}+\theta_{i,1}^{(1)}\theta_{1,j}^{(1)}}{a_{1,1}^{(1)} + \theta_{1,1}^{(1)}}.$$
Since $|a_{i,1}^{(1)}|,|a_{1,j}^{(1)}|\le |a_{1,1}^{(1)}|-2 \varepsilon <|a_{1,1}^{(1)}|- \varepsilon \le |a_{1,1}^{(1)} + \theta_{1,1}^{(1)}|$, we have
$$|\theta^{(2)}_{i,j}| \le \varepsilon \bigg(1 + \frac{2|a_{1,j}^{(1)}| +|a_{i,1}^{(1)}+ \theta_{i,1}^{(1)}|}{|a_{1,1}^{(1)} + \theta_{1,1}^{(1)}|}\bigg) \le 4 \varepsilon .$$
Repeating this estimate for $k = 3,...,n$ with $\varepsilon$ replaced by $4^{k-2} \varepsilon$, we have 
$|a_{i,j}^{(k)}-b_{i,j}^{(k)}| \le 4^{k-1} \varepsilon$ for all $i,j,k$. Suppose that $g(A)$ is achieved by entry $a_{i,j}^{(k)}$. Then
$$g(B) \ge \frac{|a^{(k)}_{i,j}|-4^{k-1} \varepsilon}{|a_{1,1}|+\varepsilon} = g(A) - \frac{4^{k-1}\varepsilon |a_{1,1}| + \varepsilon |a_{i,j}^{(k)}|}{(|a_{1,1}|+\varepsilon)|a_{1,1}|} \ge g(A) - \varepsilon(4^{n-1} +g(A))/|a_{1,1}|.$$
\end{proof}

Combining Lemmas \ref{lm:cp_eps_stable} and \ref{lm:cp_eps_ball}, we are now prepared to prove a lemma regarding the maximum growth factor over sets that cover $\mathbb{R}^{n \times n}$ (or $\mathbb{C}^{n \times n}$) sufficiently well.

\begin{lemma}\label{lm:cp_eps_cover}
Let $n>1$, $0<\varepsilon<2^{-(2n-1)}$, and $S$ equal $\mathbb{R}$ or $\mathbb{C}$. Let $X\subset S^{n \times n}$ be a subset such that, for all $A \in \mathbf{CP}_n(S)$, there exists an $\alpha \in S$ and $B \in X$ satisfying $|a_{i,j} - \alpha \, b_{i,j}| \le \varepsilon |a_{1,1}|$ for all $i,j = 1,...,n$. Then $ \mathbf{CP}_n(S) \cap X $ is non-empty and 
$$g\big[  \mathbf{CP}_n(S) \cap X \big]\ge \big(1 - \varepsilon n 4^{n-1} g\big[\mathbf{CP}_n(S)\big] /(2\varepsilon;4)_{n} \big) \, g\big[\mathbf{CP}_n(S)\big],$$
where $(\cdot;\cdot)_n$ is the q-Pochhammer symbol.
\end{lemma}

\begin{proof}
The main idea of the proof is as follows. We consider a matrix $A \in \mathbf{CP}_n(S)$, $a_{1,1} = 1$, that maximizes growth factor (i.e., $g(A) = g\big[\mathbf{CP}_n(S)\big]$) and, using Lemma \ref{lm:cp_eps_stable} applied to $\mathbf{CP}_n(S)$ and $\mathbf{CP}^{\bm{\delta}}_n(S)$ for $\bm{\delta}$ entry-wise negative, find a nearby matrix $C \in \mathbf{CP}^{\bm{\delta}}_n(S)$. Then, we find a matrix $B \in X$ nearby $C$ and, using Lemma \ref{lm:cp_eps_ball}, conclude that $B \in \mathbf{CP}_n(S)$. Finally, using the bounds on $|a^{(k)}_{i,j}-c^{(k)}_{i,j}|$ and $|b_{i,j}-c_{i,j}|$ we argue that $g(B)$ is fairly large.

So that we may apply Lemma \ref{lm:cp_eps_ball}, we define $\delta_k = - 2 \times 4^{k-1} \varepsilon$ and let $C \in \mathbf{CP}^{\bm{\delta}}_n(S)$ be the matrix resulting from the proof of Lemma \ref{lm:cp_eps_stable}. Because $A$ maximizes $g(A)$, $|a^{(k)}_{k,k}| \ge 1$ for $k=1,...,n$ and therefore $|c^{(k)}_{k,k}| \ge 1$ for $k=1,...,n$ as well. In this case, $C$ satisfies
$$|c_{i,j}^{(k)}| \le (1+ \delta_k) |c_{k,k}^{(k)}| = |c_{k,k}^{(k)}|-2 \times 4^{k-1} \varepsilon |c_{k,k}^{(k)}| \le |c_{k,k}^{(k)}|-2 \times 4^{k-1} \varepsilon,$$
and so, by Lemma \ref{lm:cp_eps_ball} combined with our lemma hypothesis, there exists a matrix $B \in \mathbf{CP}_n(S) \cap X$ (w.l.o.g. $\alpha=1$) with $|b_{i,j}-c_{i,j}| \le \varepsilon$.

What remains is to bound the differences $|a_{1,1} - b_{1,1}|$ and $|a_{n,n}^{(n)}-b_{n,n}^{(n)}|$, and compute a lower bound for $g(B)$. By Lemmas \ref{lm:cp_eps_stable} and \ref{lm:cp_eps_ball},
\begin{align*} 
|a_{1,1} - b_{1,1}| &\le |a_{1,1} - c_{1,1}| + |b_{1,1} - c_{1,1}| \\
&\le g(A) \bigg[ \frac{\big(1-2\times 4^{n-2} \epsilon\big)^{-2} -1}{\prod_{p=1}^{n-2} (1-2 \times 4^{p-1} \varepsilon)} + \sum_{m=1}^{n-2}  \frac{2\times 4^{m-1} \varepsilon }{\prod_{p=1}^{m} (1-2 \times 4^{p-1} \varepsilon)}  \bigg] +  \epsilon \\
&= \varepsilon \bigg(1 + g(A) \bigg[ \frac{2 \times 4^{n-2}}{1-2 \times 4^{n-2} \varepsilon} (2 \varepsilon;4)_{n-1}^{-1} + \sum_{m=1}^{n-1} 2\times 4^{m-1} (2 \varepsilon;4)_{m}^{-1}  \bigg] \bigg) \\
&\le \varepsilon \big(1 + 2 \, n \, 4^{n-2} g(A)/ (2 \varepsilon;4)_{n}  \big),
\end{align*}
and
$$|a_{n,n}^{(n)}-b_{n,n}| \le |a_{n,n}^{(n)} - c_{n,n}^{(n)}| + |b_{n,n}^{(n)} - c_{n,n}^{(n)}| \le 4^{n-1} \varepsilon.$$
Therefore,
\begin{align*}
   \frac{|b_{n,n}^{(n)}|}{|b_{1,1}|} &\ge \frac{g(A)- 4^{n-1} \varepsilon}{1+\varepsilon \big(1 + 2 \, n \, 4^{n-2} g(A)/ (2 \varepsilon;4)_{n}  \big) } \\
    &= g(A) - \varepsilon \; \frac{4^{n-1} + g(A)\big(1 + 2 \, n \, 4^{n-2} g(A)/ (2 \varepsilon;4)_{n}  \big)}{1+\varepsilon \big(1 + 2 \, n \, 4^{n-2} g(A)/ (2 \varepsilon;4)_{n}  \big)} \\
    &\ge  g(A) - \varepsilon \; \big( 4^{n-1} + g(A)\big(1 + 2 \, n \, 4^{n-2} g(A)/ (2 \varepsilon;4)_{n}  \big) \big) \\
    &= g(A) \big(1 - \varepsilon (4^{n-1}/g(A) + 1 + 2 \, n \, 4^{n-2} g(A)/ (2 \varepsilon;4)_{n}  ) \big) \\
    &\ge g(A) \big(1 - \varepsilon n 4^{n-1} g(A) /(2\varepsilon;4)_{n} \big).
\end{align*}
\end{proof}

The requirement on the cover that $X$ provides in the previous lemma is quite strong; for a non-trivial result we require $\varepsilon$ to be exponentially small in $n$. However, the set of $m \times m$ matrices $A^{(n-m+1)}$ resulting from many steps of Gaussian elimination applied to the set of $A \in S^{n \times n}$, for $S$ finite, does indeed provide an approximation to any $m \times m$ matrix with error exponentially small in $m$ (where $m$ is a sufficiently small fixed polynomial in $n$). We formalize this concept in the proof of the following theorem, which relates the maximum growth of $\mathbf{CP}_{m}(S)$, $S \subset \mathbb{R}$, $|S|>1$, to that of $\mathbf{CP}_{n}(\mathbb{R})$.

\begin{theorem}\label{thm:restrict}
If $S \subset \mathbb{R}$, then
$$g\big[\mathbf{CP}_{m}(S)\big] \ge \frac{\text{diam(S)}}{2 \max_{s \in S} |s|} g\big[\mathbf{CP}_{n}(\mathbb{R})\big] \quad \text{for all } m > 4n(3n+1).$$
\end{theorem}

\begin{proof}
The main idea of the proof is to build a matrix $B \in S^{m \times m}$, $m = n + p$, such that iterates $B^{(i)}$, $i = 1,...,p$, are completely pivoted, $|b_{p+1,p+1}^{(p+1)}|\ge |b_{1,1}|$, and $B^{(p+1)}$ approximates an arbitrary $A \in \mathbf{CP}_{n}(\mathbb{R})$ exponentially well. If we can approximate an arbitrary $A$ up to error $2^{-3n}$, i.e., $|a_{i,j} - \alpha b^{(k+1)}_{i,j}| \le 2^{-3n} |a_{1,1}|$ for some fixed $\alpha$, then, by Lemma \ref{lm:cp_eps_cover} combined with Wilkinson's bound (Inequality \ref{eqn:wilk}) for $g\big[\mathbf{CP}_{n}(\mathbb{C})\big]$ (for $n>1$),
$$g\big[\mathbf{CP}_{m}(S)\big] \ge (1-2^{-(n+1)} n^{\ln(n)/4 + 3/2}/(2^{1-3n};4)_{n})g\big[\mathbf{CP}_{n}(\mathbb{R})\big] \ge \frac{1}{2} g\big[\mathbf{CP}_{n}(\mathbb{R})\big].$$
What remains is to construct the matrix $B$. 

Given any $s_1,s_2 \in S$, $|s_1|<|s_2|$, and matrix $C \in \mathbf{CP}_{m-1}(\{0,1\})$, the matrix
$$ B = \begin{pmatrix} s_2 & s_2 \mathbf{1}^T \\ s_2 \mathbf{1} & s_2 \mathbf{1}\mathbf{1}^T + (s_1 - s_2) C \end{pmatrix}$$
is in $\mathbf{CP}_{m}(S)$ and satisfies $B^{(2)} = (s_1 - s_2) C$. Therefore, we may assume that $S = \{0,1\}$ at the cost of one step of Gaussian elimination and a multiplicative factor of $\text{diam}(S)/\max_{s\in S} |s|$ in the growth factor. However, we would like a matrix with entries in $\{0,1/2,1\}$. To do so, we note that three steps of Gaussian elimination applied to the $(m-1) \times (m-1)$ block matrix
$$C = \begin{pmatrix} 1 & 1 & 0 & \mathbf{0}^T & 0 \\ 1 & 0 & 1 & \mathbf{0}^T & 0 \\ 0 & 1 & 1 & \mathbf{0}^T & 1 \\ \mathbf{0} & \mathbf{0} & x & E & y \end{pmatrix}$$
where $x \in \{0,1\}^{m-4}$, $E \in \{0,1\}^{(m-4)\times (m-5)}$, and $y \in \{0,1\}^{m-4}$, produces a $(m-4)\times (m-4)$ matrix with its first $m-5$ columns given by $E$ and its last column given by $y-x/2$. Performing this operation $\ell$ times produces a $\ell \times \ell$ $\{0,1/2,1\}$ matrix, where $\ell$ must be such that $4\ell +1 \le m$. We are now prepared to approximate an arbitrary matrix $A \in \mathbf{CP}_{n}(\mathbb{R})$ using matrices in $\mathbf{CP}_\ell(\{0,1/2,1\})$. Suppose (w.l.o.g.) that $a_{1,1} = 1$, and let $r_{i,j,k}$ denote the $k^{th}$ bit in the binary expansion of $\text{ceil}(a_{i,j})-a_{i,j}$ (we write $-1$ as $-0.\bar{1}$ in binary), and set $r_{i,j,0}$ to be the integer part of $a_{i,j}$ (i.e., either $0$ or $1$). To obtain an approximation of $A$ of order $2^{-3n}$, we set $\ell = 3n^2 +n$ and define $E$ as follows
$$ E = \begin{pmatrix} I & \mathbf{0} & \cdots & \mathbf{0} & \tfrac{1}{2} I \\
\tfrac{1}{2} I & I &\ddots & \vdots & \tfrac{1}{2} I \\ \vdots & \ddots & \ddots &  \mathbf{0} & \vdots \\ \tfrac{1}{2} I & \cdots & \tfrac{1}{2} I & I & \tfrac{1}{2} I \\ R_{1} & R_2 & \cdots & R_{3n} & R_0 
\end{pmatrix},$$
where each block is $n \times n$, and $R_k = (r_{i,j,k})_{i,j=1}^n$ for $k = 0,1,...,3n$. After $n$ steps of Gaussian elimination, we have
$$ E^{(n+1)} = \begin{pmatrix} I & \mathbf{0} & \cdots & \mathbf{0} & \tfrac{1}{4} I \\
\tfrac{1}{2} I & I &\ddots & \vdots & \tfrac{1}{4} I \\ \vdots & \ddots & \ddots &  \mathbf{0} & \vdots \\ \tfrac{1}{2} I & \cdots & \tfrac{1}{2} I & I & \tfrac{1}{4} I \\ R_{2} & R_3 & \cdots & R_{3n} & R_0-\tfrac{1}{2} R_1 
\end{pmatrix},$$
and finally, after $3n^2$ steps we have that $E^{(3n^2+1)} = R_0 - \tfrac{1}{2} R_1 - \tfrac{1}{4} R_2 - ... - \tfrac{1}{2^{3n}} R_{3n}$ and approximates $A$ up to error $2^{-3n}$. We have $\ell = 3n^2 +n$ and require $4\ell +1 \le m$, so we set $m = 4n(3n+1)+1$.
\end{proof}

A similar result (with a worse multiplicative constant) holds for $\mathbb{C}$ given a set $S$ which either contains $\{0,1,i\}$, or can be converted to such a set after relatively few iterates of Gaussian elimination (e.g., $\{-1,1,1+i\}$). We leave the details to the motivated reader.

\section{Growth Factor in Floating Point Arithmetic}\label{sec:float}

In this section, we aim to bound the growth factor encountered in practice in floating point arithmetic. The term ``growth factor'' in the literature 
is used ambiguously to refer to two closely
related quantities: growth factor under exact arithmetic or under
floating point arithmetic, leading to some confusion. The exact case is clear, and shows up in theoretical discussions.  The floating point arithmetic case,
by contrast refers to the largest element (in absolute value) seen during a floating point computation. As previously mentioned in Section \ref{sec:intro}, error estimates for Gaussian elimination typically involve the growth factor under floating point arithmetic rather than exact arithmetic. In this section, we show that when using sufficiently high precision ($\omega(\log^2 n)$ bits), the maximum growth factor for exact and floating point arithmetic are identical up to a $1+o(1)$ multiplicative factor (Theorem \ref{thm:float}).

We consider the maximum growth factor when performing Gaussian elimination in base $\beta$ with $t$ digits of precision. For simplicity, we ignore issues of overflow and underflow. Here, we focus exclusively on real-valued matrices, but the analogous theorem for complex matrices follows quickly from the below analysis by simply adjusting the error due to multiplication and division for a given base and mantissa. We leave further details to the interested reader. Under floating point arithmetic, the procedure of Gaussian elimination is given by
\begin{align*}
    \hat a_{i,j}^{(1)}&:= a_{i,j} (1+\phi_{i,j}^{(0)}) \qquad \qquad \qquad \qquad \qquad \quad \; \; \text{for} \quad i,j = 1,...,n, \\
    \hat a^{(k+1)}_{i,j} &:= \big[\hat a^{(k)}_{i,j} - s_{i,k}  \hat a^{(k)}_{k,j} (1+\theta_{i,j}^{(k)}) \big](1+\phi_{i,j}^{(k)}) \quad \text{for} \quad   i,j = k,...,n, \; k = 1,...,n-1.
\end{align*}
where
$$s_{i,k} =  \frac{\hat a^{(k)}_{i,k} }{\hat a^{(k)}_{k,k}}(1+ \varphi_{i,k} ),$$
and $|\theta_{i,j}^{(k)}|,|\phi_{i,j}^{(k)}|,|\varphi_{i,k}| \le u:= \beta^{1-t}/2$ for all $i,j,k$ ($u$ is commonly referred to as the unit round-off). When partial pivoting is employed, we may assume that $|s_{i,k}| \le 1$ and $|s_{i,k} (1+\theta_{i,j}^{(k)})| \le 1$ for all $i,j,k$. Similar to the sets $\mathbf{CP}_n(S)$ and $\mathbf{PP}_n(S)$ defined in Section \ref{sec:intro}, we define 
\begin{align*}
    \widehat{\mathbf{CP}}_n(S) &= \{ A \in \textbf{GL}_n(\mathbb{C}) \cap S^{n \times n} \, | \, \hat a_{k,k}^{(k)} \ne 0 \text{ for all }k, \, |\hat a_{i,j}^{(k)}| \le |\hat a_{k,k}^{(k)}| \text{ for all } i,j \ge k \}, \\
    \widehat{\mathbf{PP}}_n(S) &= \{ A \in \textbf{GL}_n(\mathbb{C}) \cap S^{n \times n} \, | \, \hat a_{k,k}^{(k)} \ne 0 \text{ for all }k, \, |\hat a_{i,k}^{(k)}| \le |\hat a_{k,k}^{(k)}| \text{ for all } i \ge k \}.
\end{align*}
To avoid a proliferation of indices, here and in what follows the dependence of the above sets and the growth factor on $\beta$ and $t$ is implicit. We note that, for any partially pivoted matrix, we may assume that $|s_{i,k}| \le 1$ for all $i,k$. The growth factor under finite arithmetic is denoted by
$$G(A) := \frac{\max_{i,j,k} |\hat a_{i,j}^{(k)}|}{\max_{i,j} |\hat a^{(1)}_{i,j}|},$$
and we define $G[\mathbf{X}]$ to be the maximum growth factor under finite arithmetic (with base $\beta$ and length $t$ mantissa) over all matrices in $\mathbf{X}$. The quantity $G[\mathbf{X}]$ is a key ingredient in stability theorems of Gaussian elimination (see \cite[Theorem 2.6]{isaacson2012analysis} or \cite[Theorem 9.5]{higham2002accuracy}). In general, the best known bounds for partial, rook, and complete pivoting is given by
\begin{equation}\label{eqn:finite_trivial}
    G(A) \le \big[ 1 + (1 + u)^2\big]^{n-1} = 2^{n-1} + O(n u),
\end{equation}
and when $\beta =2$, this bound can simply be replaced by $2^{n-1}$ (see \cite[Section 1.2]{isaacson2012analysis} for details). For rook and complete pivoting, $2^{n-1}$ is much more pessimistic than Inequalities \ref{eqn:wilk} and \ref{eqn:rook_bound} for exact arithmetic. 

As the mantissa length $t$ tends to infinity, intuitively, the maximum growth factor under floating point arithmetic will converge to its exact arithmetic counterpart. However, given a single matrix, the growth factor in floating point can be very different from exact arithmetic due to ``near ties" causing the elimination to follow a different branch. That branch, however, is the exact branch of some nearby matrix, as the following lemma illustrates (for partial pivoting).

\begin{lemma}\label{lm:float}
For every $A \in \widehat{\mathbf{PP}}_n(\mathbb{R})$, there exists a matrix $B \in \mathbf{PP}_n(\mathbb{R})$ with $b_{i,j}^{(k)} = \hat a_{i,j}^{(k)}$ for $i =k$ or $j = k$, and
$$\big| \hat a_{i,j}^{(k)} - b_{i,j}^{(k)} \big| \le  u \, \sum_{\ell = k}^{\min\{i,j\}-1}\bigg[  |\hat a^{(\ell)}_{i,j}| +   |\hat a^{(\ell)}_{\ell,j}| (3 + u ) \bigg]$$
for all $i,j = k,...,n$ and $k = 1,...,n-1$.
\end{lemma}

\begin{proof}
The main idea is to iteratively update the lower right block of each matrix $\hat A^{(k)}$ so that successive matrices agree exactly, i.e., $B^{(k+1)} = B^{(k)}_{k+1:n} - B^{(k)}_{k+1:n,k} B^{(k)}_{k,k+1:n} / b^{(k)}_{k,k}$. To this end, we iteratively define $B$ so that $B^{(n)} = \hat A^{(n)}$ and
$$ B^{(k)} = \begin{pmatrix} \hat a_{k,k}^{(k)} & \hat A^{(k)}_{k,:} \\ \hat A^{(k)}_{:,k} & B^{(k+1)} + \hat A^{(k)}_{k+1:n,k} \hat A^{(k)}_{k,k+1:n} / \hat a^{(k)}_{k,k} \end{pmatrix} \quad \text{for } k = 1,...,n-1.   $$
Clearly, successive iterates of $B$ agree with each other, and $b_{i,j}^{(k)} = \hat a_{i,j}^{(k)}$ for $i = k$ or $j = k$. What remains is to bound the error in the lower right block. Consider the entry $b_{i,j}^{(k)}$, where $i,j>k$ and let $m = \min\{i,j\}$. We have
\begin{align*}
    b_{i,j}^{(k)} &= \hat a_{i,j}^{(m)} + \sum_{\ell = k}^{m-1} \hat a_{i,\ell}^{(\ell)} \hat a_{\ell,j}^{(\ell)}/ \hat a_{\ell, \ell}^{(\ell)} \\
    &= \big[\hat a^{(m-1)}_{i,j} - s_{i,m-1}  \hat a^{(m-1)}_{m-1,j} (1+\theta_{i,j}^{(m-1)}) \big](1+\phi_{i,j}^{(m-1)}) \\ &\quad + \hat a^{(m-1)}_{m-1,j} \big(s_{i,m-1}  - \varphi_{i,m-1} \hat a_{i,m-1}^{(m-1)} / \hat a_{m-1, m-1}^{(m-1)}\big) + \sum_{\ell = k}^{m-2} \hat a_{i,\ell}^{(\ell)} \hat a_{\ell,j}^{(\ell)}/ \hat a_{\ell, \ell}^{(\ell)} \\
   &= \bigg[\phi_{i,j}^{(m-1)} \hat a^{(m-1)}_{i,j} -  s_{i,m-1}  \hat a^{(m-1)}_{m-1,j} (\theta_{i,j}^{(m-1)} + \phi_{i,j}^{(m-1)} + \theta_{i,j}^{(m-1)}\phi_{i,j}^{(m-1)})  \\
    &\quad - \hat a^{(m-1)}_{m-1,j}  \varphi_{i,m-1} \hat a_{i,m-1}^{(m-1)} / \hat a_{m-1, m-1}^{(m-1)}\bigg]  + \bigg[\hat a_{i,j}^{(m-1)} + \sum_{\ell = k}^{m-2} \hat a_{i,\ell}^{(\ell)} \hat a_{\ell,j}^{(\ell)}/ \hat a_{\ell, \ell}^{(\ell)} \bigg].
\end{align*}
Repeating this procedure, we have
$$ b_{i,j}^{(k)} = \hat a_{i,j}^{(k)} + \sum_{\ell = k}^{m-1}\bigg[ \phi_{i,j}^{(\ell)} \hat a^{(\ell)}_{i,j} -  s_{i,\ell}  \hat a^{(\ell)}_{\ell,j} (\theta_{i,j}^{(\ell)} + \phi_{i,j}^{(\ell)} + \theta_{i,j}^{(\ell)}\phi_{i,j}^{(\ell)}) - \hat a^{(\ell)}_{\ell,j}  \varphi_{i,\ell} \hat a_{i,\ell}^{(\ell)} / \hat a_{\ell, \ell}^{(\ell)}\bigg].$$
Because our matrix is partially pivoted, $|s_{i,\ell}|$ and $|\hat a_{i,\ell}^{(\ell)}|/|\hat a_{\ell, \ell}^{(\ell)}|$ are at most one, and so
$$\big| \hat a_{i,j}^{(k)} - b_{i,j}^{(k)} \big| \le u \sum_{\ell = k}^{\min\{i,j\}-1}\bigg[  |\hat a^{(\ell)}_{i,j}| +   |\hat a^{(\ell)}_{\ell,j}| (3 + u) \bigg].$$
\end{proof}

By combining the above lemma with Lemma \ref{lm:cp_eps_stable}, we obtain a bound on growth factor for complete pivoting.

\begin{theorem}\label{thm:float}
Let $0<C<1$ and 
$$t \ge 1 +  \log_\beta\bigg[ \frac{(1+C)(4+5C)}{C} \sum_{m=1}^{n-1} \sum_{\ell=1}^{n-m} g\big[\mathbf{CP}_\ell(\mathbb{R})\big] g\big[\mathbf{CP}_m(\mathbb{R})\big] \bigg].$$
Then $G\big[\widehat{\mathbf{CP}}_n(\mathbb{R})\big] \le (1+C) \,  g\big[\mathbf{CP}_n(\mathbb{R})\big]$.
\end{theorem}

\begin{proof}
Suppose $A \in \widehat{\textbf{CP}}_n(\mathbb{R})$ maximizes growth, i.e., $G(A) = G\big[\widehat{\textbf{CP}}_n(\mathbb{R})]$, and let $B \in \mathbf{PP}_n(\mathbb{R})$ be a matrix satisfying $b_{k,k}^{(k)} = \hat a_{k,k}^{(k)}$ for all $k$, and the bounds of Lemma \ref{lm:float}. Then $B \in \textbf{CP}^{\bm{\varepsilon}}_n(\mathbb{R})$, $\bm{\varepsilon}=(\varepsilon_1,...,\varepsilon_{n-1})$, for $\epsilon_k := (4+u)u \sum_{\ell = 1}^{n-k} G\big[\widehat{\textbf{CP}}_\ell(\mathbb{R})]$, as
  $$  | b_{i,j}^{(k)}| \le | \hat a_{i,j}^{(k)} | + u \, \sum_{\ell = k}^{n-1}\bigg[  |\hat a^{(\ell)}_{i,j}| +   |\hat a^{(\ell)}_{\ell,j}| (3 + u ) \bigg] \le | b_{k,k}^{(k)} | \bigg(1 + u \, \sum_{\ell = k}^{n-1} \frac{|\hat a^{(\ell)}_{i,j}|}{| \hat a_{k,k}^{(k)} |} +   \frac{|\hat a^{(\ell)}_{\ell,j}|}{|\hat a_{k,k}^{(k)} |} (3 + u )  \bigg).$$
In addition, $G\big[\widehat{\textbf{CP}}_n(\mathbb{R})] = G(A) = |b_{n,n}^{(n)}|/|b_{1,1}|$. Using Lemma \ref{lm:cp_eps_stable} applied to $B$, we can find a matrix $C \in g\big[\mathbf{CP}_n(\mathbb{R})\big]$ that satisfies $b_{n,n}^{(n)} = c_{n,n}^{(n)}$ and $$|c_{1,1}| =  \bigg(1 + \max_{\ell \le n-1} \varepsilon_\ell(2+\varepsilon_\ell) | \hat a_{\ell,\ell}^{(\ell)} |/|\hat a_{1,1}|+ \sum_{m=1}^{\ell-1} \varepsilon_m | \hat a_{m,m}^{(m)}|/|\hat a_{1,1}|\bigg).$$
For the sake of space, we define $\gamma:=(4+u)u$, $g(n):= g\big[\mathbf{CP}_n(\mathbb{R})\big]$, and $G(n):=G\big[\widehat{\textbf{CP}}_n(\mathbb{R})]$, and note that
\begin{align}
    G(n) &\le g(n) \bigg( 1 + \max_{\ell \le n-1} \varepsilon_\ell(2+\varepsilon_\ell) G(\ell)+ \sum_{m=1}^{\ell-1} \varepsilon_m G(m) \bigg) \nonumber  \\
    &\le g(n) \bigg( 1 + \gamma \, \max_{\ell \le n-1} G(\ell) \bigg[1+\gamma \sum_{p=1}^{n-\ell} G(p)\bigg] \sum_{p=1}^{n-\ell} G(p) + \sum_{m=1}^{\ell} \sum_{p=1}^{n-m} G(p) G(m) \bigg) \label{ineq:float} \\
    &\le g(n) \bigg( 1 + \gamma \,  \bigg[2+\gamma \sum_{\ell=1}^{n-1} G(\ell) \bigg] \sum_{m=1}^{n-1} \sum_{p=1}^{n-m} G(p) G(m) \bigg) \nonumber.
\end{align}
The result follows from noting that if 
\begin{equation}\label{ineq:assumpt}
\frac{1}{\gamma} \ge \frac{(1+C)^2}{C} \big(2 + C/2(1+C)\big) \sum_{m=1}^{n-1} \sum_{\ell=1}^{n-m} g(\ell) g(m)
\end{equation}
for some $C>0$, then $G(k) \le (1+C) g(k)$ for all $k = 1,...,n$. Indeed, we have $G(1) = g(1)$, and, assuming $G(\ell) \le (1+C) g(\ell)$ for $\ell = 1,...,k$,
\begin{align*}
    \frac{G(k+1)}{g(k+1)} &\le 1 + \gamma (1+C)^2 \,  \bigg[2+\gamma (1+C) \sum_{\ell=1}^{n-1} g(\ell) \bigg] \sum_{m=1}^{n-1} \sum_{p=1}^{n-m} g(p) g(m) \\
    &\le  1 + \gamma (1+C)^2 \,  \bigg[2+\frac{C}{2(1+C)}\bigg] \sum_{m=1}^{n-1} \sum_{p=1}^{n-m} g(p) g(m)  \le 1+C.
\end{align*}
\end{proof}

The above theorem is incredibly pessimistic, but nevertheless still provides useful some useful information. First, by using Wilkinson's bound, we note that $G\big[\widehat{\mathbf{CP}}_n(\mathbb{R})\big] \le (1+1/\text{Poly(n)}) \,  g\big[\mathbf{CP}_n(\mathbb{R})\big]$ for $t = \omega(\log_\beta^2(n))$, and, under the assumption that $g\big[\mathbf{CP}_n(\mathbb{R})\big]$ is bounded by a polynomial, only a $t = \omega(\log_\beta(n))$ length mantissa is required. In Table \ref{tab:prec}, we include a number of possible bounds on the growth factor (including Wilkinson's Inequality \ref{eqn:wilk}), and list lower bounds on the largest value of $n$ for which Theorem \ref{thm:float} guarantees that $G\big[\widehat{\mathbf{CP}}_n(\mathbb{R})\big] \le (3/2) \,  g\big[\mathbf{CP}_n(\mathbb{R})\big]$.

\begin{table}[t]
\caption {\label{tab:prec} Lower bounds for the largest $n$ satisfying the conditions of Theorem \ref{thm:float} for IEEE 754 Double and Quadruple precision, with $C = 1/2$ (i.e., $G\big[\widehat{\mathbf{CP}}_n(\mathbb{R})\big] \le (3/2) \,  g\big[\mathbf{CP}_n(\mathbb{R})\big]$) under various assumptions on upper bounds for $g\big[\mathbf{CP}_n(\mathbb{R})\big]$.} 

\begin{tabular}{r | r | r  }
$g\big[\mathbf{CP}_n(\mathbb{R})\big] \le ...$ & Double Prec. ($t=52$) & Quadruple Prec. ($t=112$)   \\ \hline
 $3n$ &  $4188$ & $137266926$ \\  \hline
 $n^2/2$ &  $660$ & $676504$ \\  \hline
 Inequality \ref{eqn:wilk} &  $554$ & $29563$ \\  \hline
\end{tabular}

\end{table}

\section{Computer-Assisted Lower Bounds}\label{sec:numerics}

In this section, we detail lower bounds for growth factor found using computer search, and discuss how such computer-generated matrices in finite arithmetic lead to mathematically provable lower bounds for growth factor in exact arithmetic (Theorem \ref{thm:cp_growth_restate}).

\subsection{Computer-Assisted Lower Bounds for Small Dimension}\label{sec:comp}

We are indebted to the early pioneering numerical optimization given by Day \& Peterson
\cite{day1988growth}   and Gould \cite{gould1991growth}.
We are the beneficiary of more readily usable quality software (JuMP \cite{DunningHuchetteLubin2017}), the ready availability of faster processors, and also modern parallel 
computing.

Our methodology is to run 64 threads each with a random $n \times n$ starting matrix of standard normals
which has rows and columns permuted so that the matrix is completely pivoted. 
We then normalize by dividing by the $(1,1)$ element.
Our optimization
is over the $1+2^2+\ldots+n^2$ elements that are seen by Gaussian Elimination as suggested by Gould \cite{gould1991growth}.
Therefore the starting point requires using all of these $\approx n^3/3$ elements. We store the variables in a 3-d array $x_{i,j,k}, 1\le k \le n,  k \le i,j \le n$. Thus $k=1$ is the original matrix, $k=2$ is the $n-1 \times n-1$ matrix
obtained after one step of Gaussian elimination.

We include a listing of the high level function from the 
\href{https://github.com/alanedelman/CompletePivotingGrowth}{online repository}
\cite{ourrepo}
that performs the optimization in Figure \ref{fig:algorithm}; it is quite easy
to read as it is similar to the mathematics.  We invite readers to note the six lines of code that indicate
the nonlinear constraints  (\verb+@NLconstraint+)  and linear constraints  (\verb+@constraint+), the first one of which is the constraint of Gaussian elimination:

\begin{figure}
{\footnotesize \jlinputlisting{run_model.jl}}
\caption{Our run\_model function performs growth optimization from a random start.  We encourage the reader to examine the constraints: they correspond to the mathematical constraints a completely pivoted matrix satisfies and are easy to read. The variable $x$ is a three-dimensional array that stores the Gaussian elimination
``pyramid,” e.g., $x[(i,j,k)]$ is the $(i,j)^{th}$ entry of the $k^{th}$ step of Gaussian elimination and the $(k,k,k)$ entry is the $k^{th}$ pivot.}\label{fig:algorithm}
\end{figure}

Very importantly we also wish to discuss the line towards the bottom that begins \verb+B = convert_to_cp(Rational+...
as this line turns a floating point answer to a rigorous mathematical answer.
A simple observation is that an output of optimization software is not yet a theoretical lower bound
because of floating point effects.  In particular it is possible that
an output of a program is not completely pivoted.
The examples from Gould \cite{gould1991growth,gould1991private} were close in the floating 
point sense to being optimums and some minor tweaking was needed \cite{edelman1992complete} 
for the purpose of exact mathematics.  In 
 \cite{edelman1992complete}  the first author asked if there would always be a nearby floating point matrix.
 In this paper, we show that Lemma \ref{lm:cp_eps_stable} theoretically states there would be nearby matrix,
 and Algorithm \ref{alg:eps_stable} as embodied in the \verb+convert_to_cp+  function working on a rational
 form allows us to state that our computer assisted solutions constitute exact rigorous mathematics
 rather than a floating point approximation.

For the smaller values of $n$, we tend to believe that the the lower bounds found may well be close to $g\big[\mathbf{CP}_n(\mathbb{R})\big]$
as we have on occasion rerun these values, and found the same answers.  For larger values of $n$, we imagine that
the lower bounds are just that, lower bounds.

\begin{figure}[t]
\begin{centering}
\includegraphics[width=3.5in]{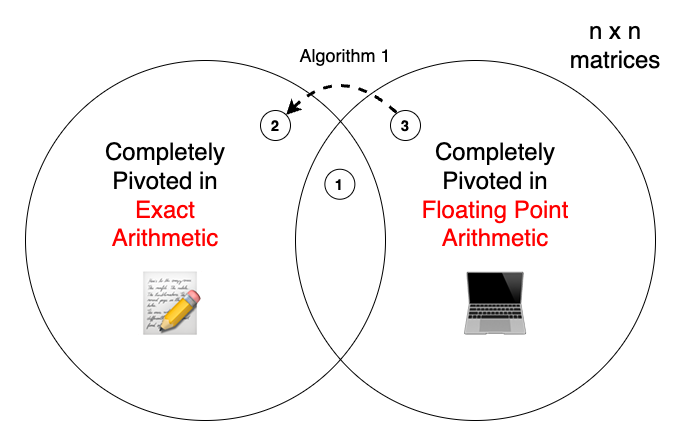}
\caption{The above figure shows that we can go from matrices that are completely pivoted in floating point to matrices that are completely pivoted in exact arithmetic. 
Lemma \ref{lm:cp_eps_stable} proves that this is possible and Algorithm \ref{alg:eps_stable} provides a pseudocode implementation 
(a Julia implementation may be found in the \href{https://github.com/alanedelman/CompletePivotingGrowth}{online repository} 
\cite{ourrepo}). For instance, Algorithm \ref{alg:eps_stable} has fully automated Edelman's exact arithmetic extension of Gould's finite precision counterexample to Conjecture \ref{conj:cryer}, and provides some answers to Edelman's perturbation question for growth factor \cite{edelman1992complete}.}
\label{fig:theorypractice}
\end{centering}
\end{figure}

\subsection{Provable Lower Bounds}
Next, we prove global lower bounds for the growth factor under complete pivoting by converting estimates for growth factor for small $n$ into estimates for all $n$. We begin with the following extrapolation lemma.

\begin{lemma}\label{lm:cp_extrap}
Let $S$ equal $\mathbb{R}$ or $\mathbb{C}$. Then
\begin{enumerate}[(i)]
    \item $g\big[\mathbf{CP}_n(S)\big]$ is non-decreasing,
    \item $g\big[\mathbf{CP}_{2n}(S)\big] \ge 2 \, g\big[\mathbf{CP}_{n}(S)\big]$ for all $n \in \mathbb{N}$,
    \item if $g\big[\mathbf{CP}_n(S)\big] \ge C n$ for $n =k,...,2k-1$, then
$g\big[\mathbf{CP}_n(S)\big] \ge \frac{(1/k; 1/2)_{\infty}}{1-1/k} \, C n $ for all $n \ge k$, where $(\cdot;\cdot)_\infty$ is the q-Pochhammer symbol.
\end{enumerate}
\end{lemma}

\begin{proof}
Properties $(i)$ and $(ii)$ follow simply from the operations
$$\begin{pmatrix} 1 & 0_n^T \\ 0_n & A \end{pmatrix} \qquad \text{and} \qquad A \otimes H_1, \text{ where } H_1 := \begin{pmatrix} 1 & 1 \\ 1 & -1 \end{pmatrix},$$
applied to a matrix $A \in \mathbf{CP}_n(S)$, respectively (Property $(ii)$ is also proved in \cite{tornheim1970maximum}).
If $g\big[\mathbf{CP}_n(S)\big] \ge C \, n$ for all $n \in [k,2k)$, then by Properties $(i)$ and $(ii)$,
$$g\big[\mathbf{CP}_{2n+1}(S)\big] \ge g\big[\mathbf{CP}_{2n}(S)\big] \ge 2 C n =\frac{2n}{2n+1} \,  C (2n+1) \ge \frac{2k}{2k+1} \,  C (2n+1)$$
for all $n \in [k,2k)$, i.e., $g\big[\mathbf{CP}_n(S)\big] \ge \frac{2k}{2k+1} \, C n$ for all $n \in [k,4k)$. Repeating this argument, we obtain the lower bound 
$$g\big[\mathbf{CP}_n(S)\big] \ge C n \, \prod_{i=1}^j \frac{2^{i} k}{2^{i}k +1} \ge C n \, \prod_{i=1}^j \bigg( 1 -\frac{1}{2^i k} \bigg) = \frac{(1/k; 1/2)_{j+1}}{1-1/k} \, C n $$
for $n \in \big[k, 2^{j+1} k \big)$, where $(\cdot; \cdot)_j$ is the $q$-Pochhammer symbol. Noting that $(\cdot; \cdot)_j$ is monotonically non-increasing with respect to $j$ for non-negative inputs of magnitude at most one completes the proof of Property $(iii)$.
\end{proof}

Combining Lemma \ref{lm:cp_extrap} with the computer-assisted (and mathematically provable) lower bounds of Table \ref{tab:cp_r} immediately implies a lower bound for all values of $n$.

\begin{theorem}[Restatement of Theorem \ref{thm:cp_growth}]\label{thm:cp_growth_restate}
$g\big[\mathbf{CP}_n(\mathbb{R})\big] \ge 1.0045 \, n$ for all $n > 10$, and $\limsup_n \big(g\big[\mathbf{CP}_n(\mathbb{R})\big] /n\big) \ge 2.525$.
\end{theorem}

\begin{proof}
The lower bound for all $n > 10$ follows from checking $n = 11,12,13$ by hand and applying Property $(iii)$ of Lemma \ref{lm:cp_extrap} to $k = 14$ (with $C = 1.08$). The asymptotic bound follows directly from our lower bound for $n = 100$ combined with Property $(ii)$ of Lemma \ref{lm:cp_extrap}.
\end{proof}

\section{Rook Pivoting}\label{sec:rook}

The majority of this work focuses on complete pivoting, due to its theoretical and practical importance. Rook pivoting by comparison is relatively understudied, yet the quasi-polynomial bound on growth factor combined with a reduced computational complexity compared to complete pivoting in practice makes this an attractive technique. Many of the results of this paper also apply to rook pivoting, sometimes leading to even stronger results. These details are left to the interested reader.

Through a stability lemma, tensor argument, and numerically computed lower bounds for a fixed value of $n$, we provide improved lower bounds for the maximum growth factor with rook pivoting. Let
$$\mathbf{RP}^{\bm{\varepsilon}}_n(S) = \{ A \in \textbf{GL}_n(\mathbb{C}) \cap S^{n \times n} \, | \, |a_{i,k}^{(k)}|,|a_{k,j}^{(k)}| \le (1+\varepsilon_k) |a_{k,k}^{(k)}| \text{ for all } i,j \ge k \}.$$
We have the following proposition (in the spirit of Lemma \ref{lm:cp_eps_stable}).

\begin{proposition}\label{lm:rp_stab}
For every $A \in \mathbf{RP}^{\bm{\varepsilon}}_n(S)$, where $S$ equals $\mathbb{R}$ or $\mathbb{C}$ and $\varepsilon_i \ge 0$ for $i =1,...,n-1$, there exists a matrix $B \in \mathbf{RP}_n(S)$ such that
$$a^{(k)}_{n,n} = b^{(k)}_{n,n} \qquad \text{for} \quad k = 1,...,n,$$ 
and 
$$ \big| a^{(k)}_{i,j} - b^{(k)}_{i,j} \big| \le (2  + \varepsilon_\ell)\varepsilon_\ell \, \big|a_{\ell,\ell}^{(\ell)}\big|, \quad \ell := \min\{i,j\},$$
for all $i,j = k,...,n$, $k = 1,...,n-1$.
\end{proposition}

\begin{proof}
Given $A \in \mathbf{RP}_n^{\bm{\varepsilon}}(S)$, the result follows immediately from the construction $B^{(n)} := A^{(n)}$ and 
$$B^{(k)} := \begin{pmatrix} (1+\varepsilon_k)^2 \, a^{(k)}_{k,k} & (1+\varepsilon_k) \, A^{(k)}_{k,k+1:n} \\
 (1+\varepsilon_k) \, A^{(k)}_{k+1:n,k}  & A^{(k)}_{k+1:n,k+1:n} + B^{(k+1)} - A^{(k+1)}
\end{pmatrix}$$
for $k = 1,...,n-1$. 
\end{proof}

Similar to Lemma \ref{lm:cp_eps_stable}, the construction of $B \in \mathbf{RP}_n(S)$ is algorithmic in nature, and this procedure (a variant of Algorithm \ref{alg:eps_stable}) converts inexact numerically computed instances of large growth into provable lower bounds. In particular, through the combination of numerical computation and an algorithmic implementation of the procedure of Lemma \ref{lm:rp_stab}, we have the following lower bound (see Subsection \ref{sec:comp} and 
\href{https://github.com/alanedelman/CompletePivotingGrowth}{our repository} \cite{ourrepo}
).

\begin{proposition}\label{prop:rook_lowerbound}
$g\big[\mathbf{RP}_{48}(\mathbb{R})\big] > 640.4861$.
\end{proposition}

Next, we prove the following extrapolation lemma, from which lower bounds for rook pivoting immediately follow.

\begin{lemma}\label{prop:extrap_rp}
Let $S$ equal $\mathbb{R}$ or $\mathbb{C}$. Then
\begin{enumerate}[(i)]
    \item $g\big[\mathbf{RP}_n(S)\big]$ is non-decreasing,
    \item $g\big[\mathbf{RP}_{m\, n}(S)\big] \ge  g\big[\mathbf{RP}_{m}(S)\big] \times g\big[\mathbf{RP}_{n}(S)\big]$ for all $m,n \in \mathbb{N}$,
    \item if $g\big[\mathbf{RP}_k(S)\big] \ge k^\alpha$ for some $k$, then $g\big[\mathbf{RP}_n(S)\big] \ge k^{-\alpha}n^\alpha$ for all $n \in \mathbb{N}$.
\end{enumerate}
\end{lemma}

\begin{proof}
Property $(i)$ follows from the construction $\begin{pmatrix} 1 & 0_n^T \\ 0_n & A \end{pmatrix}$. Property $(ii)$ follows from the fact that if $A \in \mathbf{RP}_{m}(S)$, $B \in \mathbf{RP}_{n}(S)$, and $S$ is closed under addition and multiplication, then $A \otimes B \in \mathbf{RP}_{m \, n}(S)$, where $\otimes$ is the matrix Kronecker product, which we now prove. Let $C = A \otimes B$, and, for the sake of space, define the following three auxilary matrices, consisting of $B^{(k)}$ for some $k = 2,...,n$ and some zeros:
$$B_r^{(k)} = \begin{pmatrix} 0_{n-k+1,k-1} & B^{(k)} \end{pmatrix}, \quad B_c^{(k)} = \begin{pmatrix} 0_{k-1,n-k+1} \\ B^{(k)} \end{pmatrix}, \quad B_f^{(k)} = \begin{pmatrix} 0_{k-1,k-1} & 0_{k-1,n-k+1} \\ 0_{n-k+1,k-1} & B^{(k)} \end{pmatrix},$$
so that $B_r^{(k)} \in S^{n \times (n-k+1)}$, $B_c^{(k)} \in S^{(n-k+1) \times n}$, and $B_f^{(k)} \in S^{n \times n}$. It suffices to complete $n$ steps of Gaussian elimination, show that at each step the rook pivoting condition holds ($|c_{k,k}^{(k)}| \ge |c_{i,k}^{(k)}|,|c_{k,j}^{(k)}|$ for $k = 1,...,n$), and note that $C^{(n+1)} = A^{(2)} \otimes B$. Initially, we have
$$C^{(1)}= A \otimes B = \begin{pmatrix} a_{1,1}B & \cdots & a_{1,m} B \\ \vdots & \ddots & \vdots \\ a_{m,1} B & \cdots & a_{m,m} B \end{pmatrix},$$
and the rook pivoting condition holds initially for any Kronecker product $A \otimes B$ of rook pivoted matrices $A$ and $B$, as
$$|a_{1,1} b_{1,1}| = |a_{1,1}| \, |b_{1,1}| \ge \max_{i,j = 1,...,m}\{|a_{i,1}|,|a_{1,j}|\} \; \max_{i,j = 1,...,n}\{|b_{i,1}|,|b_{1,j}|\}.$$
On the $k^{th}$ step of Gaussian elimination, we have
$$C^{(k)} = \begin{pmatrix} a_{1,1} B^{(k)} & a_{1,2} B_r^{(k)} & \cdots & a_{1,m}B_r^{(k)} \\ a_{2,1} B_c^{(k)} & a_{2,2}^{(2)} B + (a_{2,2}-a_{2,2}^{(2)}) B_f^{(k)} & \cdots & a_{2,m}^{(2)} B + (a_{2,m}-a_{2,m}^{(2)}) B_f^{(k)} \\ \vdots & \vdots & \ddots & \vdots \\ a_{m,1} B_c^{(k)} & a_{m,2}^{(2)} B + (a_{m,2}-a_{m,2}^{(2)}) B_f^{(k)} & \cdots & a_{m,m}^{(2)} B + (a_{m,m}-a_{m,m}^{(2)}) B_f^{(k)}\end{pmatrix},$$
and still the rook pivoting condition holds, as both $A$ and $B^{(k)}$ are rook pivoted. Finally, after the $n^{th}$ step, we note that the remainder term $(a_{i,j}-a_{i,j}^{(2)}) B_f^{(n)}$ disappears,
as
$$(a_{i,j}-a_{i,j}^{(2)}) B_f^{(n)} - \frac{a_{i,1} a_{1,j}}{a_{1,1}} \frac{B_c^{(n)} \, B_r^{(n)}}{b_{n,n}^{(n)}} = 0_{n \times n},$$
and so $C^{(n+1)} = A^{(2)} \otimes B$.

Property $(iii)$ follows quickly from Properties $(i)$ and $(ii)$. Let $n > k$ (if $n \le k$, the result trivially holds), and let $\ell \in \mathbb{N}$ be the largest number such that $k^\ell \le n$. We have
$$g\big[\mathbf{RP}_{n}(S)\big] \ge g\big[\mathbf{RP}_{k^\ell}(S)\big] \ge k^{\alpha \ell} =  \big[k^{\ell}/n \big]^{\alpha} \, n^{\alpha} \ge k^{-\alpha} n^{\alpha}.$$
\end{proof}

Using Proposition \ref{prop:rook_lowerbound} and Lemma \ref{prop:extrap_rp}, we obtain our desired lower bound.

\begin{theorem}[Restatement of Theorem \ref{thm:rp_growth}]
$g\big[\mathbf{RP}_{n}(\mathbb{R})\big] > \frac{1}{641}n^{1.669}$ for all $n \in \mathbb{N}$.
\end{theorem}

\section*{Acknowledgements}

\setstretch{.4}
\begin{onehalfspacing}

{\tiny  This material is based upon work supported by the Institute for Advanced Study and the National Science Foundation under Grant No. DMS-1926686. The authors acknowledge the MIT SuperCloud and Lincoln Laboratory Supercomputing Center for providing HPC resources that have contributed to the research results reported within this paper. This material is based upon work supported by the National Science Foundation under grant no. OAC-1835443, grant no. SII-2029670, grant no. ECCS-2029670, grant no. OAC-2103804, and grant no. PHY-2021825. We also gratefully acknowledge the U.S. Agency for International Development through Penn State for grant no. S002283-USAID. The information, data, or work presented herein was funded in part by the Advanced Research Projects Agency-Energy (ARPA-E), U.S. Department of Energy, under Award Number DE-AR0001211 and DE-AR0001222. The views and opinions of authors expressed herein do not necessarily state or reflect those of the United States Government or any agency thereof. This material was supported by The Research Council of Norway and Equinor ASA through Research Council project ``308817 - Digital wells for optimal production and drainage''. Research was sponsored by the United States Air Force Research Laboratory and the United States Air Force Artificial Intelligence Accelerator and was accomplished under Cooperative Agreement Number FA8750-19-2-1000. The views and conclusions contained in this document are those of the authors and should not be interpreted as representing the official policies, either expressed or implied, of the United States Air Force or the U.S. Government. The U.S. Government is authorized to reproduce and distribute reprints for Government purposes notwithstanding any copyright notation herein. We further thank Juan-Pablo Vielma, for very helpful discussions. In addition, we would like to thank Nick Higham for useful comments on an earlier draft and are deeply saddened by his passing.}
\end{onehalfspacing}

\nocite{*}
{ \small 
	\bibliographystyle{plain}
	\bibliography{main} }

\end{document}